\documentclass[10pt, a4paper]{amsart}
\usepackage[T1]{fontenc}
\usepackage{geometry}
\usepackage[dvipsnames]{xcolor}
\usepackage[hidelinks,colorlinks=true,citecolor=ForestGreen,linkcolor=RubineRed]{hyperref}
\usepackage[timezonesep={\ UTC}, showseconds=false]{datetime2}
\usepackage{bookmark}

\usepackage{amsmath}
\usepackage{amssymb}
\usepackage{amsfonts}
\usepackage{mathtools}
\usepackage{amsthm}
\usepackage{braket}
\usepackage{comment}
\usepackage{mathrsfs}
\usepackage{dsfont}
\usepackage{bbm}
\usepackage{stmaryrd}
\usepackage{tensor}
\usepackage{tikz-cd}
\usepackage{extarrows}
\usepackage{manfnt,bbold}

\DeclareFontFamily{U}{wncy}{}
\DeclareFontShape{U}{wncy}{m}{n}{<->wncyr10}{}
\DeclareSymbolFont{mcy}{U}{wncy}{m}{n}
\DeclareMathSymbol{\Sha}{\mathord}{mcy}{"58}

\newcommand{\f}{\mathbf{f}}
\newcommand{\Hom}{\mathrm{Hom}}
\newcommand{\cA}{\mathcal{A}}
\newcommand{\cF}{\mathcal{F}}
\newcommand{\cG}{\mathcal{G}}
\newcommand{\cN}{\mathcal{N}}
\newcommand{\cO}{\mathcal{O}}
\newcommand{\cP}{\mathcal{P}}
\newcommand{\cX}{\mathcal{X}}
\newcommand{\cL}{\mathcal{L}}

\newcommand{\cH}{\mathcal{H}}

\newcommand{\sW}{\mathscr{W}}

\newcommand{\QQ}{\mathbb{Q}}
\newcommand{\FF}{\mathbb{F}}
\newcommand{\ZZ}{\mathbb{Z}}
\newcommand{\Zp}{\mathbb{Z}_p}

\newcommand{\TT}{\mathbb{T}}
\newcommand{\Q}{\mathbb{Q}}
\newcommand{\Z}{\mathbb{Z}}

\newcommand{\fP}{\mathfrak{P}}

\newcommand{\GL}{\mathrm{GL}}
\newcommand{\ord}{\mathrm{ord}}
\newcommand{\Gal}{\mathrm{Gal}}
\newcommand{\Heeg}{\mathrm{Heeg}}

\newcommand{\Cor}{\mathrm{cor}}
\newcommand{\res}{\mathrm{res}}
\newcommand{\Sel}{\mathrm{Sel}}
\newcommand{\col}{\mathrm{Col}}
\newcommand{\tr}{\mathrm{tr}}
\newcommand{\Tr}{\mathrm{Tr}}
\newcommand{\tors}{\mathrm{tors}}
\renewcommand{\div}{\mathrm{div}}
\newcommand{\cotors}{\mathrm{cotors}}
\newcommand{\loc}{\mathrm{loc}}
\newcommand{\Char}{\mathrm{char}}
\newcommand{\im}{\mathrm{im}}
\newcommand{\rank}{\mathrm{rank}}
\newcommand{\corank}{\mathrm{corank}}
\newcommand{\coker}{\mathrm{coker}}
\newcommand{\ur}{\mathrm{ur}}

\newcommand{\PP}{\mathbb{P}}
\newcommand{\len}{\mathrm{len}}
\newcommand{\ind}{\mathrm{ind}}

\newcommand{\Aut}{\mathrm{Aut}}

\newcommand{\Dcris}{\mathbb{D}_\mathrm{cris}}
\newcommand{\Mlog}{M_{\log}}
\newcommand{\Fil}{\mathrm{Fil}}
\newcommand{\1}{\mathbb{1}}

\newcommand{\sL}{\mathscr{L}}
\newcommand{\sR}{\mathscr{R}}
\newcommand{\sT}{\mathscr{T}}
\newcommand{\sF}{\mathscr{F}}

\numberwithin{equation}{section}
\newtheorem{thm}{Theorem}[section]
\newtheorem{mainThm}{Theorem}

\newtheorem{prop}[thm]{Proposition}
\newtheorem{lem}[thm]{Lemma}
\newtheorem{cor}[thm]{Corollary}
\newtheorem{conj}[thm]{Conjecture}
\theoremstyle{remark}
\newtheorem{rem}[thm]{Remark}
\theoremstyle{definition}
\newtheorem{de}[thm]{Definition}

\usepackage{bbm}
\usepackage[bbgreekl]{mathbbol}
\DeclareMathSymbol\bka  \mathord{bbold}{"14}

\definecolor{ForestGreen}{RGB}{34,139,34}

\title{Kolyvagin's conjecture at non-ordinary primes}

\author[A.~Lei]{Antonio Lei}
\address[Lei]{Department of Mathematics and Statistics\\University of Ottawa\\
150 Louis-Pasteur Pvt\\
Ottawa, ON\\
Canada K1N 6N5}
\email{antonio.lei@uottawa.ca}

\author[L.~Zhao]{Luochen Zhao}
\address[Zhao]{Morningside Center of Mathematics\\No.55 Zhongguancun East Road\\Beijing\\100190\\China}
\email{luochenzhao@amss.ac.cn}


\begin{document}

\begin{abstract}Let $K$ be an imaginary quadratic field and let $p \ge 5$ be a prime that is unramified in $K$. Let $\mathcal{A}_f/\mathbb{Q}$ be an abelian variety of $\mathrm{GL}_2$-type associated with a weight-two modular form $f$, with good non-ordinary reduction at $p$, and suppose that $(f,K)$ satisfies the generalized Heegner hypothesis. In the case where $p$ is inert in $K$, we further assume that $\mathcal{A}_f$ is an elliptic curve.
We develop an Euler-characteristic formula for signed Selmer groups over anticyclotomic $\mathbb{Z}_p$-extensions that applies when the corresponding Selmer modules have arbitrary $\Lambda$-rank. 
Assuming one inclusion in the signed Iwasawa main conjecture, we apply this formula to prove Kolyvagin's conjecture on the non-vanishing of the Kolyvagin system attached to Heegner points. 
Our results extend to the non-ordinary setting the results of Wei Zhang, Burungale–Castella–Grossi–Skinner, Castella--Sano and Kim in the ordinary case, and complement the works of Sweeting and Kim in the non-ordinary case under different hypotheses. 
We also study the effect of the exceptional zero phenomenon on the Iwasawa main conjecture in the inert case.
\end{abstract}

\maketitle


\section{Introduction}

Throughout this article, $K$ is an imaginary quadratic field of discriminant $D_K$ and  $p\ge5$ a prime unramified in $K$. 
We fix a weight two newform $f$ of level $N_0$ and trivial nebentypus, such that $p\nmid N_0$ and $(N_{0},D_{K})=1$. 
Let $F/\Q_p$ be the local Hecke field generated by the Fourier coefficients $\{a_n(f)\}_{n\ge 1}$ of $f$, $\cO$ be its ring of integers, and $\varpi\in \cO$ a uniformizer.
Let $\cA_f$ be an associated $\GL_2$-type abelian variety over $\QQ$ so that $\cO \cong {\rm End}(\cA_f)\otimes\Zp$ and
\[
\prod L(f^\sigma,s)=L(\cA_{f},s),
\]
where the product runs over the Galois conjugates of $f$.
Let $T$ denote the $\varpi$-adic Tate module of $\cA_{f,F}$, the base change of $\cA_f$ to $F$. 

In this article, we assume that $\cA_{f,F}$ has good non-ordinary reduction at $\varpi$.
We write $N_0=N^+N^-$, where $N^+$ (resp. $N^-$) is only divisible by primes which are split (resp. inert) in $K$. We assume throughout that  
\begin{align}\tag{Heeg.}\label{hyp:Heeg}
	N^- \text{ is a square-free product of an even number of primes.}
\end{align}
Furthermore, we assume that $\cA_f$ is an elliptic curve $E$ when $p$ is inert in $K$. This assumption allows us to work with the plus and minus Selmer groups defined in \cite{BBL1,BBL2}, which rely on the local points constructed in \cite{BKO1}.

The main goal of this article is to study Kolyvagin's conjecture on the non-vanishing of the eponymous system arising from Heegner points attached to $(\cA_f,K)$. Our method is inspired by the recent work of Burungale--Castella--Grossi--Skinner \cite{bcgs} on the validity of this conjecture for $p$-ordinary elliptic curves when $p$ is split in $K$. In this article, we suppose the form $f$ has good non-ordinary reduction at $p$, and consider both cases where $p$ is either split or inert in $K$. In particular, we show that under some hypotheses, Kolyvagin's conjecture follows from one inclusion in the anticyclotomic signed Iwasawa main conjecture, which we now briefly review.

Let $K_\infty$ denote the anticyclotomic $\Zp$-extension of $K$ and put $\Gamma=\Gal(K_{\infty}/K)$. For an integer $m\ge0$, write $K_m$ for the unique subextension of $K_{\infty}/K$ such that $[K_m:K]=p^m$. 
Put $G_m=\Gal(K_m/K)$.
Let $\Lambda$ denote the anticyclotomic Iwasawa algebra $\cO[[\Gamma]]=\displaystyle\varprojlim_m\cO[G_m]$, and put $\TT = T\otimes _{\cO}\Lambda^\iota$, where the action of $G_K$ is given by $G_K\to \Gamma$ composed with the inversion map. Throughout, we fix a topological generator $\gamma$ of $\Gamma$ and identify $\Lambda$ with the power series ring $\cO[[X]]$, by sending $\gamma$ to $1+X$. In the non-ordinary setting, generalizing works on the cyclotomic $\Zp$-extension of $\QQ$ \cite{kobayashi03,sprung,LLZ14}, one can construct signed Selmer groups $\Sel_{\sharp/\flat}(K,\TT)$ and $\Sel_{\sharp/\flat}(K_\infty,A)$ as in \cite{iovita-pollack,castella-wan24,BL-21,BBL1,BBL2,BLV} (their definitions are recalled in \S\ref{sec:IMC}). Under the generalized Heegner hypothesis \eqref{hyp:Heeg}, we are endowed with the usual system of Heegner classes $\kappa^\Heeg_n$, which can be used to construct two systems of signed Heegner classes
	\[
		\bka_1^{\sharp/\flat}\in H^1(K,\TT), \ \bka_n^{\sharp/\flat}\in H^1(K,\TT/I_n\TT)
	\]
	for $n$ a squarefree product of Kolyvagin primes, and $I_n\subset \Z_p$ the ideal generated by the set $\{\ell+1,a_\ell(f)\}_{\ell\mid n}$. The signed Iwasawa main conjecture \cite{BLV,BBL2} predicts that, for $\bullet\in \{\sharp,\flat\}$, both $\Sel_\bullet(K,\TT)$ and $\Sel_\bullet(K_\infty,A)^\vee$ are of $\Lambda$-rank 1. Furthermore, 
	\begin{align}\label{eq:MC-inclusion}
		\Char_
        \Lambda\left(\Sel_\bullet(K,\TT)/\Lambda\bka_1^\bullet\right)^2\subseteq \Char_\Lambda((\Sel_\bullet(K_\infty,A)^\vee)_\tors).
	\end{align}
Note that when $p$ is inert in $K$, $a_p = 0$, $\bullet = \sharp$, and that $f$ is attached to an elliptic curve, it was observed in \cite{BLV} that the right-hand side has an ``extra zero'' at the trivial character. Consequently, one does not expect equality to hold in this case. As a by-product of our proof of Kolyvagin's conjecture, we show that this expectation is indeed correct under mild hypotheses.

    A key difficulty in the non-ordinary setting is that the signed Selmer groups that arise naturally in the anticyclotomic setting are not, in general, cotorsion over the Iwasawa algebra. To overcome this difficulty, we develop an Euler-characteristic formula for Selmer modules of arbitrary $\Lambda$-rank. This formula allows us to relate the size of signed Tate--Shafarevich groups to specializations of the characteristic ideals appearing in the signed Iwasawa main conjecture.
    
	The main results of this article are summarized below.
\begin{mainThm}[Corollary~\ref{cor:rank-sel}, Theorem~\ref{thm:kolyvagin}]\label{thmA}
    Let $p\ge 5$ be a prime, and let $K$ be an imaginary quadratic field of discriminant $D_K$ prime to $p$. Let $f$ be a weight 2 newform of level $N_0$ prime to $pD_K$ and of trivial nebentypus. Let $\bullet\in \{\sharp,\flat\}$. Assume the generalized Heegner hypothesis \eqref{hyp:Heeg} and that $f$ is rational when $p$ is inert in $K$. Assume in addition the following conditions:
    \begin{enumerate}
        \item [i)] the signed $\Lambda$-adic Heegner class $\bka^\bullet_1\in H^1(K,\TT)$ is not $\Lambda$-torsion;
        \item [ii)] if $p$ is inert in $K$, then the residual $G_K$-representation $T/\varpi T$ is absolutely irreducible.
    \end{enumerate}
    Then, the following assertions hold.
    \begin{enumerate}
        \item[(A)] Both $\Sel_\bullet(K,\TT)$ and $\Sel_\bullet(K_\infty,A)^\vee$ are of $\Lambda$-rank 1.
        
        \item[(B)] If either $p$ is split in $K$ or $\bullet = \flat$, and the the reverse inclusion to \eqref{eq:MC-inclusion} holds in $\Lambda\otimes \Q_p$, then the Kolyvagin conjecture is true; i.e., there exists some $n$ for which $\kappa_n^\Heeg \ne 0$. 

        \item[(C)] In the exceptional case where $p$ is inert in $K$ and $\bullet = \sharp$, the reverse inclusion to \eqref{eq:MC-inclusion} in $\Lambda\otimes \Q_p$ fails.
    \end{enumerate}
\end{mainThm}
\begin{rem}
    Under certain technical assumptions, hypothesis i) is known by \cite{cornut,vatsal,cornut-vatsal}; see also \cite[Corollary 6.4]{castella-wan24} when $p$ is split in $K$; hypothesis ii) is known when the attached elliptic curve is semi-stable \cite[Proposition 2.1]{edixhoven97}, as $p$ is a supersingular prime; see Lemma \ref{lem:central} and Remark \ref{rem:irreducibility}.
\end{rem}

\begin{rem}\label{rem:reverse-IMC}
    We shall refer to the case where either $p$ is split in $K$ or $\bullet = \flat$ as the \textit{non-exceptional case}.
The validity of the signed Iwasawa main conjecture in such case, and therefore the reverse inclusion to \eqref{eq:MC-inclusion}, has been established in \cite{castella-wan24,BLV,BBL2} if the following conditions hold:

\vspace{0.2cm}

\noindent\textbf{(ram)}\; If $\ell \mid N_0$,  then the residual representation $\overline{\rho}_f$  is ramified at $\ell$ in the following cases:
\begin{itemize}
    \item[-] $\ell \mid N^{+}$,
    \item[-] $\ell \mid N^{-} \text{ and } \ell^2 \equiv 1 \pmod{p}$.
\end{itemize}

\vspace{0.2cm}
\noindent\textbf{(Im)}\; If $p$ is inert in $K$, then $\cA_f$ is semi-stable; if $p$ is split in $K$, then the image of the residual representation $\overline{\rho}_f$ contains the set of matrices of $\GL_2(\cO/\varpi\cO)$ whose determinants are in $\FF_p^\times$.

\vspace{0.2cm}

\noindent\textbf{(iso)}\; If $a_p(f) \ne 0$, then $f$ is $p$-isolated.

\vspace{0.2cm}

\noindent In particular, if these hypotheses hold, Theorem~\ref{thmA} implies the validity of Kolyvagin's conjecture for modular forms under mild hypotheses. Moreover, if $p$ is split in $K$ and $f$ is attached to a semi-stable elliptic curve over $\Q$, then under some restrictions on $N$, the rational signed main conjectures hold by \cite[Theorem C]{castella-wan24}, so the Kolyvagin conjecture in this case is valid with fewer restrictions.
\end{rem}

Write
\[
    \ord(\{\kappa^\Heeg_n\}_n) = \min\{r:\text{there exists }n\text{ such that }\kappa_n^\Heeg\ne0\text{ with }\nu(n) = r\}.
\]
Suppose the form $f$ is rational and attached to an elliptic curve $E/\Q$. Denote by $\Sel_{p^\infty}(E/K)$ the discrete $p$-primary Selmer group of $E/K$, and set $r(E/K)^\pm = \corank_{\Z_p}(\Sel_{p^\infty}(E/K)^\pm)$, where the superscript signifies the eigenspace of $\Sel_{p^\infty}(E/K)$ under the action of the non-trivial element of $\Gal(K/\Q)$. The standard argument of \cite[Theorem 4]{kolyvagin91} gives the following:
\begin{cor}
    Suppose that the hypotheses in Theorem \ref{thmA} hold, that we are in the non-exceptional case, and that $f$ is attached to an elliptic curve $E/\Q$. Assume that the reverse inclusion of \eqref{eq:MC-inclusion} holds in $\Lambda\otimes \Q_p$ (e.g., when conditions in Remark \ref{rem:reverse-IMC} are met). Then 
    \[
        \ord(\{\kappa^\Heeg_n\}_n) = \max\{r(E/K)^+, r(E/K)^-\} - 1.
    \]
    In particular, we have the following $p$-converse of Gross--Zagier and Kolyvagin for the triple $(E,p,K)$ (\textit{cf.}, \cite[Theorem 6.10]{castella-wan24}, \cite[Corollary 1.5]{BBL2}):
    \[
        \corank_{\Z_p}(\Sel_{p^\infty}(E/K)) = 1 \Longrightarrow \ord_{s=1}L(s,E/K) = 1.
    \]
\end{cor}

\subsection{Strategy}
Our approach follows a strategy similar to that of \cite{bcgs}. The assumption on the signed Heegner class allows us to choose abundant characters $\alpha$ of $\Gamma$ such that the specialization $\bka_1^\bullet(\alpha)$ is non-zero. We describe the size of a certain signed Tate--Shafarevich group in terms of the evaluation of a characteristic element of $(\Sel_\bullet(K_\infty,A)^\vee)_\tors$ at $\alpha$ and some explicit local terms, in close analogy to the ``anticyclotomic control theorem'' in \cite[Theorem 1.2.7]{bcgs} that originate from \cite{greenberg-cetraro,jetchev-skinner-wan:BSD-analytic-rank-1}.

Our main point of departure from \cite{bcgs} is that we work directly with signed Selmer groups instead of the relaxed-strict one and the corresponding BDP $p$-adic $L$-function. These objects are not available when $p$ is inert in $K$. As a consequence, the above control theorem requires a more delicate analysis, since, unlike their BDP counterparts, the signed Selmer groups are not cotorsion over the Iwasawa algebra. In particular,  the standard result of Greenberg \cite[\S4]{greenberg-cetraro} becomes inaccessible here. We overcome this difficulty by developing an Euler-characteristic formula that works with Selmer modules of arbitrary $\Lambda$-ranks; see Theorem~\ref{thm:euler-char}. It may be of independent interest since formulas of this kind are usually proved only for cotorsion Selmer groups in the literature. Another contribution of this article is Lemma \ref{lem:index}, which establishes a divisibility property of the Kolyvagin system that is used to bound the Shafarevich--Tate groups, and does not appear in \cite{bcgs}. To the best of our knowledge, this lemma appears indispensable for our approach to Kolyvagin's conjecture. 

In another direction, the irreducibility of the residual representation allows us to apply the general machinery of Kolyvagin systems developed by Howard \cite{howard04,howard04-gl2} to obtain an upper bound on the size of the signed Tate--Shafarevich group (see Theorem~\ref{thm:howard}). Combining this with the information extracted from the signed Iwasawa main conjecture via the Euler-characteristic formula, we deduce Kolyvagin's conjecture.

As discussed in \cite[Remark~1.3.2]{bcgs}, some of these ideas are also present in the work of Loeffler--Zerbes \cite{LZ20}, where the Iwasawa theory of Hilbert modular forms was studied. It would be interesting to see whether the results presented in this article can be extended to that setting.

\subsection{Related work}
Since the seminal work of Zhang \cite{zhang14}, which established Kolyvagin’s conjecture in the case of $p$-ordinary elliptic curves, numerous generalizations have been obtained, extending and refining Zhang's results in a variety of settings. In \cite{sweeting}, Sweeting vastly generalized Zhang's result, removing several hypotheses. In particular, the case where $p$ is a non-ordinary prime is included. 
Other results on Kolyvagin's conjecture were also obtained by Kim \cite{kim24,kim26} under various hypotheses.
Using Hida theory, Kolyvagin's conjecture for higher weight $p$-ordinary modular forms has been studied in \cite{LPV}. In \cite{DaRonche}, Da Ronche obtained similar results for modular forms at non-ordinary primes using signed Selmer groups defined using Wach modules. These works impose specific hypotheses on the residual Galois representation associated with the modular form, typically requiring that it be either surjective or irreducible. This assumption is weakened in \cite{bcgs} in the $p$-ordinary case, which also proved the refined Kolyvagin's conjecture.
More recently, the refined Kolyvagin's conjecture in the $p$-ordinary case is further studied in \cite{CastellaSano} using Selmer complexes.
The aforementioned works, with the exception of \cite{sweeting}, assume that the prime $p$ splits in the imaginary quadratic field $K$ or $p$ is an ordinary prime. One of the novel features of the method developed in this current work is that $p$ can be either split or inert in $K$. Our results also complement those proven in \cite{sweeting} since we only impose hypotheses on the Iwasawa main conjectures related to $f$, whereas the work of Sweeting assumes the validity of the $p$-adic Birch and Swinnerton-Dyer conjecture for modular forms congruent to $f$.

\subsection*{Acknowledgement}
The authors thank Francesc Castella, Giada Grossi, Chan-Ho Kim, Paul-Antoine Seitz and Xin Wan for helpful discussions during the preparation of the article. We are indebted to Giada Grossi for clarifying a question regarding \cite{bcgs}. AL’s research was supported by the NSERC Discovery Grant RGPIN-2026-04351. Portions of this work were completed during LZ’s visit to the University of Ottawa in Fall 2025, funded by the Visiting Researchers Program – CAS and C9 League – from the Office of International Research and Experiential Learning.

\section{Anticyclotomic Iwasawa main conjectures}\label{sec:IMC}

\subsection{Signed Heegner classes}
\label{subsec:signed-heegner}
We recall the definition of the Kolyvagin classes arising from Heegner points and discuss the construction of signed versions of these classes over anticyclotomic towers.

Let $\cL_0$ be the set of rational primes $\ell$ such that $\ell$ is inert in $K$ and $\ell\nmid N_0p$. We say that $\ell\in\cL_0$ is a \textit{Kolyvagin prime} if 
$$
    M(\ell):=\min\{\ord_\varpi(\ell+1),\ord_\varpi(a_\ell)\}>0.
$$
The set of such primes is denoted by $\cL$.

Let $\cN$ be the set of square-free products of Kolyvagin primes, and for $n\in\cN$ set $M(n):=\min\{M(\ell):\ell|n\}$ if $n\ne 1$, and $M(1) = \infty$. Let $I_n=\varpi^{M(n)}\cO$. For $m\in \Z_{\ge 1}$, denote by $\cN^{(m)} = \{n\in \cN\colon M(n)\ge m\}$ and $\cN_m = \{n\in \cN\colon \nu(n) = m\}$, where $\nu(n)$ stands for the number of prime divisors of $n$.

Throughout this section, we fix an element $n\in\cN$. For an integer $k\ge0$, let $P[np^k]\in \cA_f(K[np^k])$ be the Heegner point on $\cA_f$  of conductor $np^k$ as constructed in \cite{zhang1,zhang2}; see also \cite[\S1]{howard04-gl2}. We shall regard $P[np^k]$ as an element of $H^1(K[np^k],T)$ through the Kummer map. We recall the norm relations satisfied by these elements. 
If $p$ splits in $K$, let $\sigma$ and $\sigma^*$ denote the Frobenius maps in $\cG(n) = \Gal(K[n]/K)$ of the two primes above $p$. Set $\delta$ as the order of $(\cO_K/p\cO_K)^\times/(\ZZ/p\ZZ)^\times$ and $u_K = (\#\cO_K^\times)/2$. Let $\gamma_0=a_p$ if $p$ is inert and $\gamma_0=a_p-\sigma-\sigma^*$ if $p$ splits. Let $\gamma_1=u_K^{-1} a_p\gamma_0-\delta$. For a prime $\ell\in\cN$ that is coprime to $n$, we have
\begin{equation}
\Tr_{K[\ell np^k]/K[np^k]} P[\ell np^k]=a_\ell P[np^k]. \label{eq:norm-l}
\end{equation}
Furthermore,
\begin{align}
u_K\cdot \Tr_{K[ np]/K[n]} P[np]&=\gamma_0 P[n] ,\label{eq:norm np-n}\\
\Tr_{K[ np^{k+1}]/K[np^k]} P[np^{k+1}]&=a_p P[np^k]-P[np^{k-1}], \quad k\ge1.\label{eq:norm npk}
\end{align}
Combining \eqref{eq:norm np-n} and \eqref{eq:norm npk} (with $k=1$) gives
\begin{equation}
  \Tr_{K[ np^2]/K[n]} P[np^2]=\gamma_1 P[n].\label{eq:norm-np2}
\end{equation}

\begin{de}
For an integer $k\ge1$, let $\cG_k=\Gal(K[np^{k+1}]/K[np])\cong \ZZ/p^{k}\ZZ$ (we have suppressed the dependency on $n$ for simplicity). Let $\tilde\Lambda_{k}=\cO[\cG_k]$ and $\tilde\Lambda=\varprojlim\tilde\Lambda_k$, which we identify with $\cO[[\Gamma^{(np)}]]\cong\cO[[X]]$, where $\Gamma^{(np)}=\Gal(K[np^\infty]/K[np])$. 
 We write $C_{k}$ for the matrix $\begin{bmatrix}a_p&1\\-\Phi_{k}&0\end{bmatrix}$, where $\Phi_k\in\tilde\Lambda$ is the $p^k$-th cyclotomic polynomial corresponding to $\frac{(1+X)^{p^k}-1}{(1+X)^{p^{k-1}}-1}$ under the above identification.
We write $H_{k}$ for the $\tilde\Lambda$-morphism
\begin{align*}
\tilde\Lambda_k^2&\longrightarrow\tilde\Lambda_k^2\\
\begin{bmatrix}
x\\ y
\end{bmatrix}&\longmapsto C_{k}\cdots C_{1}\begin{bmatrix}
x\\ y
\end{bmatrix}
,
\end{align*}
which we identify with the $2\times 2$ matrix $C_k\cdots C_1$ valued in $\tilde\Lambda_k$. When $k=0$, we set $H_0$ to be the identity map.
\end{de}

We recall that there is an isomorphism $\tilde\Lambda^2\cong \varprojlim \tilde\Lambda_k^2/\ker H_k$, as explained in \cite[(3.3)]{BBL1}.

\begin{thm}
\label{thm_main_sharpflat_Heegner_ES}
Let $n\in\cN$. For any positive integer $k\ge1$, we have a unique pair of cohomology classes 
$$\begin{bmatrix}
P[np^{k+1}]^\sharp\\ P[np^{k+1}]^\flat 
\end{bmatrix}\,\,\in\,\, H^1(K[np^{k+1}],T)^{\oplus 2}/\ker(H_{k})\cdot H^1(K[np^{k+1}],T)^{\oplus 2}$$ 
satisfying the following properties:
\item[i)] We have
$$
    H_{k}
    \begin{bmatrix}
    P[np^{k+1}]^\sharp \\ P[np^{k+1}]^\flat
    \end{bmatrix}= \begin{bmatrix}
    P[np^{k+1}]\\ -\res_{K[np^{k+1}]/K[np^k]}\left( P[np^k]\right)
    \end{bmatrix},
$$
where the equality takes place in $H^1(K[np^{k+1}],T)^{\oplus 2}$.

\item[ii)] We have the containment
$${\Cor}_{K[np^{k+2}]/K[np^{k+1}]}\,\begin{bmatrix}
P[np^{k+2}]^\sharp \\ P[np^{k+2}]^\flat
\end{bmatrix} -\begin{bmatrix}
P[np^{k+1}]^\sharp \\ P[np^{k+1}]^\flat
\end{bmatrix}\,\, \in \,\, \ker(H_{k}) \cdot H^1(K[np^{k+1}],T)^{\oplus 2}\,.$$
\end{thm}
\begin{proof}
This follows form \eqref{eq:norm np-n} using the same proof as \cite[Theorem~4.10]{ABCL}.    
\end{proof}

Let $\tilde\TT = T\otimes_{\cO} \tilde\Lambda^\iota$, where $\tilde\Lambda^\iota$ denotes the $G_{K[np]}$-module, which is $\tilde\Lambda$ as a set, with the action of $G_{K[np]}$ given by the inverse of its image in $\Gamma^{(np)}$. Similarly, put $\TT=T\otimes_{\cO}\Lambda^\iota$, where $\Lambda^\iota$ denotes the $G_K$-module $\cO[[\Gamma]]$ with the natural $G_K$-action reversed.

    Through Theorem~\ref{thm_main_sharpflat_Heegner_ES}, we can define $\tilde \PP_n^\sharp,\tilde \PP_n^\flat\in H^1(K[np],\tilde\TT)$ as the classes given by the inverse limit
    \[
    \left(\begin{bmatrix}
        P[np^{k+1}]^\sharp\\ P[np^{k+1}]^\flat
    \end{bmatrix}\right)_{k\ge1}\in\varprojlim_k H^1(K[np^{k+1}],T)^{\oplus 2}/\ker (H_k)\cdot H^1(K[np^{k+1}],T)^{\oplus 2}.
    \]

Next, we extend the definition of $P[np^k]^{\sharp/\flat}$ to include the cases where $k = 0, 1$.

For $\bullet\in \{\sharp,\flat\}$, we define $P[np]^{\bullet}\in H^1(K[np],T)$ as $\Cor_{K[np^2]/K[np]}P[np^2]^\bullet$. Similarly, set $P[n]^{\bullet} = \Cor_{K[np^2]/K[n]}P[np^2]^\bullet \in H^1(K[n],T)$. Note that under the natural projection $H^1(K[np],\tilde\TT)\to H^1(K[np], T)$, $\tilde\PP_n^\bullet$ is mapped to $P[np]^\bullet$.

\begin{lem}\label{lem:specializing-Heegner}
    The elements $P[n]^\sharp$ and $P[n]^\flat$ coincide with $\gamma_0 P[n]$ and $-\delta P[n]$. respectively.
\end{lem}
\begin{proof}
Setting $k=1$ in Theorem~\ref{thm_main_sharpflat_Heegner_ES}.i) gives
    \[
\begin{bmatrix}
    a_p&1\\-\Phi_1&0
\end{bmatrix}    
\begin{bmatrix}
P[np^2]^\sharp \\ P[np^2]^\flat
\end{bmatrix}= \begin{bmatrix}
P[np^2]\\ -\res_{K[np^2]/K[np]}\left( P[np]\right)
\end{bmatrix}.
    \]
    Applying $\Cor_{K[np^2]/K[n]}$, we deduce from \eqref{eq:norm np-n} and \eqref{eq:norm-np2}:
     \[
\begin{bmatrix}
    a_p&1\\-p&0
\end{bmatrix}    
\begin{bmatrix}
P[n]^\sharp \\ P[n]^\flat
\end{bmatrix}=\begin{bmatrix}
    \gamma_1 P[n]\\ -p\gamma_0 u_K^{-1} P[n]
\end{bmatrix}.
    \]
   Thus,
   \[
\begin{bmatrix}
P[n]^\sharp \\ P[n]^\flat
\end{bmatrix}= \begin{bmatrix}
\gamma_0\\ \gamma_1-a_p\gamma_0
\end{bmatrix}P[n]= \begin{bmatrix}
\gamma_0\\ -\delta
\end{bmatrix}P[n],
    \]
    as desired.
\end{proof}

We define $\PP_n^\bullet$ as the image of $\tilde\PP_n^\bullet$ under the composition of maps 
\[
H^1(K[np],\tilde\TT)\to H^1(K[np],\TT) \xrightarrow{\rm cores} H^1(K[n],\TT),
\]
where the first map is induced from the inclusion $\Gal(K[np^\infty]/K[np])\to \Gal(K_\infty/K)$.
\begin{rem}\label{rem:untilde-P}
    We note that under the projection $H^1(K[n],\TT) \to H^1(K[n],T)$ induced by the trivial character $\1: \Gamma\to \{1\}$, $\PP_n^\bullet$ is sent to $P[n]^\bullet$ for $\bullet \in \{\sharp,\flat\}$. This is because the following diagram is commutative:
    \[
        \begin{tikzcd}
            H^1(K[np],\tilde \TT) \ar[d] \ar[r] & H^1(K[n], \TT) \ar[d]\\
            H^1(K[np],T) \ar[r,"\rm cores"] & H^1(K[n], T),
        \end{tikzcd}
    \]
    where the top arrow is the map defining $\PP_n^\bullet$.
\end{rem}

Let $D_n$ be Kolyvagin's derivative. That is,
\[
D_n=\prod_{\ell |n}\left(\sum_{j=1}^\ell j\sigma_\ell ^j\right),
\]
where $\sigma_\ell$ is a generator of the cyclic group $G(\ell):=\Gal(K[\ell]/K[1])$.

\begin{de}
Let $\cG(n)=\Gal(K[n]/K)$, $G(n)=\Gal(K[n]/K[1])\simeq\prod_{\ell|N} \Gal(K[\ell]/K[1]) \subseteq \cG(n)$, and fix $S_n$ a set of coset representatives of $\cG(n)/G(n)$.
We set $\tilde\kappa^\Heeg_n=\sum_{\sigma\in S_n}\sigma D_n(P[n])\in H^1(K[n],T)$.

For $\bullet\in\{\sharp,\flat\}$, we define $\tilde \bka_n^\bullet\in H^1(K[n],\TT)$ as the image of $ \PP_n^\bullet$ under the map $\sum_{\sigma\in S_n}\sigma D_n$.
\end{de}

\begin{lem}\label{lem:T-no-invariant}
    Let $w'$ be a prime of $K[np^\infty]$ lying above $p$. We have $H^0(K[np^\infty]_{w'}, T/\varpi) = 0$.
\end{lem}
\begin{proof}
    Let $w$ be the prime of $K[n]$ lying below $w'$. Then $K[n]_w/\QQ_p$ is an unramified extension, and $K[np^\infty]_{w'}/K[n]_w$ is a totally ramified $\Zp$-extension. Let $I$ and $I'$ be the inertia groups of $K[n]_w$ and $K[np^\infty]_{w'}$, respectively. It follows from \cite[Theorem~2.6]{edixhoven92} that the action of $I$ on $T/\varpi T$ is given by nontrivial characters of order exactly $p^2-1$. As $I/I'$ is isomorphic to the direct sum of a pro-$p$ group with $(\cO/p)^\times/(\Z/p)^\times$, we have $H^0(I',T/\varpi T) =H^0(I,T/\varpi T) =0$. Hence, $H^0(K[np^\infty],T/I_n T)=0$, as desired.
\end{proof}

\begin{cor}\label{cor:iso-rest}
    The restriction map induces an isomorphism
    \[
    H^1(K,\TT/I_n\TT)\cong H^1(K[n],\TT/I_n\TT)^{\cG(n)}.
    \]
    In addition, the images of $\tilde \bka_n^\sharp,\tilde \bka_n^\flat$ under the natural map $H^1(K[n],\TT)\to H^1(K[n],\TT/I_n\TT)$ belong to $H^1(K[n],\TT/I_n\TT)^{\cG(n)}$. Similarly, we have an isomorphism
    \[
        H^1(K,T/I_n T)\cong H^1(K[n],T/I_n T)^{\cG(n)},
    \]
    and the image of $\tilde{\kappa}_n^\Heeg$ in $H^1(K[n],T/I_n)$ is invariant under $\cG(n)$.
\end{cor}
\begin{proof}

By the inflation-restriction exact sequence, to prove the first isomorphism, it suffices to show the vanishing of $H^0(K[n],\TT/I_n \TT)$. Since the latter is a subgroup of $H^0(K[np^\infty],\TT/I_n \TT) = H^0(K[np^\infty],T/I_n T)\otimes \Lambda^{\iota}$, the desired vanishing follows from Lemma \ref{lem:T-no-invariant}.

The assertion regarding the $\cG(n)$-invariance of $\tilde \bka_n^{\sharp/\flat}$ follows from the same proof as \cite[Lemma~1.7.1]{howard04}. The rest of the statement can be proved in a manner similar to the discussion above (see, e.g., \cite[\S\S3-4]{gross-durham}).
\end{proof}

\begin{de}
We define $\kappa_n^\Heeg\in H^1(K,T/I_nT)$ as the unique class that maps to the Kummer image of $\tilde \kappa_n^\Heeg$. 
    
Similarly, we define $\bka_n^\sharp,\bka_n^\flat\in H^1(K,\TT/I_n
    \TT)$ to be the classes that are sent to the images of $\tilde \bka_n^\sharp,\tilde \bka_n^\flat$ in $H^1(K[n],\TT/I_n \TT)^{\cG(n)}$ under the isomorphism given by Corollary~\ref{cor:iso-rest}.
\end{de}

\begin{de}
 Given a character $\alpha:\Gamma\to\cO^\times$, 
 put $T_\alpha=T\otimes\cO(\alpha)$, $V_\alpha = V\otimes_{F}F(\alpha)$ and $A_\alpha=V_\alpha/T_\alpha$. We denote the natural specialization maps by
    \[
    \pi_\alpha:H^1(K,\TT/I_n\TT)\to H^1(K,T_\alpha/I_nT_\alpha),\quad H^1(K_w,\TT/I_n\TT)\to H^1(K_w,T_\alpha/I_nT_\alpha).
    \]
We define for $\bullet\in\{\sharp,\flat\}$ the twisted Heegner class $\bka_n^\bullet(\alpha)\in H^1(K,T_\alpha/I_n
    T_\alpha)$ as the image of $\bka_n^\bullet$ under $\pi_\alpha$.
\end{de}

For a character $\alpha$ of $\Gamma$ satisfying $ \alpha\equiv1\mod \varpi^m\cO\cap \Z_p$ and $n\in\cN$ such that $M(n)\ge m$, there is a natural map
\[
H^1(K,T/I_nT)\to H^1(K,T/\varpi^mT)=H^1(K,T_\alpha/\varpi^mT_\alpha).
\]
We can compare the specialization of $\bka_n^\bullet(\alpha)$ with $\kappa_n^\Heeg$ modulo $\varpi^m$ via this map.

\begin{lem}\label{lem:cong-class-alpha}
    Suppose $\alpha$ is a character of $\Gamma$ with $\alpha\equiv 1\mod \varpi^m$. For all $n\in\cN$ with $M(n)\ge m$, we have the congruence
    \[
    \bka_n^\sharp(\alpha)\equiv \gamma_0 \kappa_n^\Heeg,\quad\bka_n^\flat(\alpha)\equiv -\delta \kappa_n^\Heeg
    \]
    as elements in $H^1(K,T_\alpha/\varpi^m T_\alpha)= H^1(K,T/\varpi^m T)$.
\end{lem}
\begin{proof}
By Remark \ref{rem:untilde-P}, applying $\sum_{\sigma\in S_n}\sigma D_n$ to the equations given by Lemma~\ref{lem:specializing-Heegner} gives
\[
    \tilde\bka_n^\sharp(\mathbb{1})=\gamma_0\tilde\kappa_n^\Heeg,\quad \tilde\bka_n^\flat(\mathbb{1})=-\delta\tilde\kappa_n^\Heeg
\]
as elements in $H^1(K[n],T)$.
As in $M(n)\ge m$ and $\alpha\equiv 1\mod \varpi^m$, comparing the images of these classes in $H^1(K,T_\alpha/\varpi^m T_\alpha) = H^1(K,T/\varpi^m T)$ gives the congruences asserted.
\end{proof}

\begin{rem}\label{rem:kappa-sharp}
    Note that Lemma~\ref{lem:cong-class-alpha} tells us that $\bka_n^\sharp(\alpha)\equiv 0\mod p^m$ in the case where $p$ is inert in $K$ and $a_p=0$. This is related to the ``extra zero'' phenomenon observed in \cite{BLV}. 
\end{rem}

\subsection{Signed conditions and Coleman maps}
We review the signed local conditions at $p$ that are used to define signed Selmer groups. We also recall orthogonality properties satisfied by these conditions.

Let $w$ be a place in $K$ above $p$. In what follows, we write $\Lambda'=\Lambda\otimes_{\Z_p}\cO_{K_w}$, $F'=F\otimes_{\Q_p}K_w$, and $\cO'$ the valuation ring of $F'$. When $p$ splits in $K$, we can identify $K_w$ with $\Q_p$, and we have $\Lambda'=\Lambda$, 
$F'=F$ and $\cO'=\cO$. In the case where $p$ is inert in $K$, we have $F=\Q_p$ and $K_w=\QQ_{p^2}$, so $\Lambda'=\ZZ_{p^2}[[\Gamma]]$, $F'=\QQ_{p^2}$ and $\cO'=\ZZ_{p^2}$. We recall from \cite[\S4.1]{BBL1} that we can construct signed Coleman maps
\[
\col_{[n]}^\sharp,\col_{[n]}^\flat: H^1(K_w,\TT/I_n\TT)\rightarrow \Lambda'/I_n
\]
when given a $Q$-system of elements in the sense of Definition 4.1 of \textit{op.~cit}. When $n=1$, we simply write 
\[
\col^\sharp,\col^\flat: H^1(K_w,\TT)\rightarrow \Lambda'.
\]

\begin{de}
For $\bullet\in\{\sharp,\flat\}$ and a prime $w|p$, we define $H^1_\bullet(K_w,\TT/I_n\TT)$ as the kernel of $\col_{[n]}^\bullet$.
\end{de}

We conclude this subsection by establishing the self-orthogonality of $\ker(\col^\bullet)$.

\subsubsection{The split case}\label{S:orth-split}

We review the construction of the Coleman maps when $p$ is split in $K$, as outlined in \cite[\S5]{BBL1}. Let $w$ be a place in $K$ above $p$. Let $\cL_{w}:H^1(K_w,\TT)\rightarrow\Dcris(T)\otimes\cH_\sW(\Gamma)$ be the Perrin-Riou map, which is defined as the specialization of the two-variable Perrin-Riou map in \cite{LZ0} (see \cite[Theorem~5.1]{CastellaHsiehGHC}). Here, $\cH_\sW(\Gamma)$ denotes the algebra of tempered distributions on $\Gamma$ with coefficients in $\sW$, where $\sW$ is the ring of integers of the completion of the maximal unramified extension of $\Q_p$.

Let $\Mlog=\displaystyle\lim_{m\to\infty}\begin{bmatrix}
    a_p&1\\-p&0
\end{bmatrix}^{-m-1}C_m\cdots C_1$. Then
\[
\cL_w(z)=u_e\begin{bmatrix}
    v_1&v_2
\end{bmatrix}\Mlog\begin{bmatrix}
    \col^\sharp(z)\\\col^\flat(z)
\end{bmatrix},
\]
where $v_1$ is an $\cO$-basis of $\Fil^0\Dcris(T)$, $v_2=\varphi(v_1)$, and $u_e$ is a unit in $\Lambda\hat\otimes\sW$.

Denote by $T^*$ the linear dual $\Hom_{\cO}(T,\cO)$. Note that $T\cong T^*(1)$ as $G_\Q$-representations. Let $\{v_1',v_2'\}$ be the dual basis of $\{v_1,v_2\}$ with respect to the standard pairing $[\sim,\sim]:\Dcris(T)\times\Dcris(T^*(1))$. Note that $v_2'$ is an $\cO$-basis of $\Fil^0\Dcris(T)$ and $v_1'=-\varphi(v_2')$ since $[v_1,v_1]=[v_2,v_2]=0$ and $[v_1,v_2]=-[v_2,v_1]\ne0$. As explained in \cite[\S3]{lei-ponsinet}, we have the dual Coleman maps satisfying
\[
\cL_w(z)=u_e\begin{bmatrix}
    v_1'&v_2'
\end{bmatrix}\Mlog'\begin{bmatrix}
    \col'^\flat(z)\\\col'^\sharp(z)
\end{bmatrix},
\]
where $\Mlog'=\displaystyle\lim_{m\to\infty}\begin{bmatrix}
   0& p\\-1&a_p
\end{bmatrix}^{-m-1}\begin{bmatrix}
    0&\Phi_m\\-1&a_p
\end{bmatrix}\cdots\begin{bmatrix}
    0&\Phi_1\\-1&a_p
\end{bmatrix}$. In addition, for $\bullet\in\{\sharp,\flat\}$,  $\ker\col^\bullet$ and $\ker\col'^\bullet$ are orthogonal complement of each other with respect to the $\Lambda$-adic pairing 
\[
\langle\sim,\sim\rangle:H^1(K_w,\TT)\times H^1(K_w,\TT)\to\Lambda
\]
as proved in Lemma 3.2 of \textit{op.~cit}. Furthermore, the bases $\{v_1,v_2\}$ and $\{v_1',v_2'\}$ are related by a change of basis matrix
\[
\begin{bmatrix}
    v_1& v_2
\end{bmatrix}=\begin{bmatrix}
    v_1'&v_2'
\end{bmatrix}\begin{bmatrix}
    0&-c\\
    c&0
\end{bmatrix},
\]
where $c$ is a unit of $\cO$ since both $v_1$ and $v_2'$ are bases of $\Fil^0\Dcris(T)$, and $v_2=\varphi(v_1)$ and $v_1=-\varphi(v_2)$. Therefore, it follows from the discussion of \cite[Lemma~2.16]{kazim-antonio17} that
\[
\begin{bmatrix}
    \col'^\flat\\ \col'^\sharp
\end{bmatrix}=\begin{bmatrix}
    0&-c\\
    c&0
\end{bmatrix}\begin{bmatrix}
    \col^\sharp\\ \col^\flat
\end{bmatrix},
\]
and hence
\[
\ker\col^\bullet=
\ker\col'^\bullet,\quad \bullet\in\{\sharp,\flat\}.
\]
In particular, we deduce that $\ker\col^\bullet$ is self-orthogonal under $\langle\sim,\sim\rangle$.

\subsubsection{The inert case}\label{S:orth-inert}
When $p$ is inert in $K$, recall that we assume that $\cA_f$ is an elliptic curve and $a_p=0$. In this case, we can define $H^1_\bullet(K_w,\TT)$ explicitly in terms of certain special local points on $\cA_f$ constructed in \cite{BKO1,BKO2}. Let $\hat\cA_f$ denote the formal group of $\cA_f$ at $p$. Let $w$ be the unique place of $K$ lying above $p$. Let $w'$ be a place of $K_\infty$ lying above $w$. For $m\ge0$, let $k_m$ denote the subextension of $K_{\infty,w'}/K_w$ of degree $p^m$.

\begin{thm}\label{thm:Q-inert}
There exists a system of local points $d_m\in \hat{\cA_f}({k_m})$ such that:
\begin{itemize}
    \item[(1)] $\Tr_{k_m/k_{m-1}}d_m=-d_{m-2}$  for all $m\ge 2$;
    \item[(2)] $\Tr_{k_1/k_0}d_1=-d_{0}$;
    \item[(3)] $d_0\in \hat{\cA_f}({k_0})\setminus p\hat{\cA_f}({k_0})$.
\end{itemize}
\end{thm}
\begin{proof}
This is \cite[Theorem~5.5]{BKO1}. 
\end{proof}

Define
\[
d_m^+=\begin{cases}
d_m&\textrm{if $m$ is even},\\
d_{m-1}&\textrm{if $m$ is odd},
\end{cases}\qquad d_m^-=\begin{cases}
d_{m-1}&\textrm{if $m\ge2$ is even},\\
d_{m}&\textrm{if $m$ is odd}.
\end{cases}\quad
\]
Let $\hat \cA_f^\pm(k_m)$ be the $\Lambda'$-modules generated by $d_m^\pm$. These modules can be described in terms of the trace maps, as in \cite[Definition~8.16]{kobayashi03}.
We may regard $\hat \cA_f^\pm(k_m)$ as subgroups of $H^1_\f(k_m,T)$ via the Kummer map.

We define $\hat \cA_f^\pm(k_m)^\perp\supset H^1_\f(k_m,T)$ to be the orthogonal complement of $\hat \cA_f^\pm(k_m)$ under the local Tate pairing. Then the kernels of $\col^{\sharp/\flat}$ can be realized as $\displaystyle\varprojlim_m \hat \cA_f^\pm(k_m)^\perp$ respectively (see \cite[Proposition~4.30]{BBL1}).  In particular, $\ker(\col^\bullet)$ is self-orthogonal for $\bullet\in\{\sharp,\flat\}$ following \cite[Proposition~3.15]{kim07} and \cite[Theorem~2.9]{kim14};  see also \cite[\S8.1]{kataoka} and \cite[Lemma~5.6]{BLV}.

\subsection{Signed Selmer groups and Kolyvagin systems}
In this section, we recall the definition of signed Selmer groups. We also introduce signed Shafarevich--Tate groups, similar to those studied in \cite{HLV,lei-fine}.
\begin{de}
If $\ell\in\cN$ and $M$ is a $G_K$-representation, we define
\[
H^1_\tr(K_\lambda,M)=\ker\left(H^1(K_\lambda,M)\to H^1(K[\ell]_{\lambda'},M)\right),
\]
where $\lambda$ is the unique place of $K$ lying above $\ell$, and $\lambda'$ is a  prime of $K[\ell]$ above $\lambda$. Also recall that, if $w\nmid p$, we have the Bloch--Kato finite subgroup $H^1_\f(K,M)$, the singular quotient $H^1_{\rm s}(K,M)$ \cite[\S1.1]{howard04}, and a finite-singular isomorphism for $\ell\in \cL$ and $M$ a subquotient of $\TT$:
\[
  \phi^{\rm fs}_\ell: H^1_\f(K_\ell,M)\to H^1_{\rm s}(K_\ell,M).
\]

For $\bullet\in\{\sharp,\flat\}$, we define the \textit{signed Selmer group} $H^1_{\cF^\bullet(n)}(K,\TT/I_n\TT)$ as
\[
    \ker\left(H^1(K,\TT/I_n\TT)\to\prod_{w}\frac{ H^1(K_w,\TT/I_n\TT)}{H^1_{\cF^\bullet(n)}(K_w,\TT/I_n\TT)}\right),
\]
where
\[
    H^1_{\cF^\bullet(n)}(K_w,\TT/I_n\TT) = 
    \begin{cases}
        H^1_\bullet(K_w,\TT/I_n\TT) & w|p;\\
        H^1_\tr(K_w,\TT/I_n\TT) & w|n;\\
        H^1_\f(K_w,\TT/I_n\TT) & w\nmid np.
    \end{cases}
\]
Furthermore, if $\sT$ is a quotient (resp.~subrepresentation) of $\TT/I_n\TT$, we propagate the Selmer condition $\cF^\bullet(n)$ to $\sT$ by dictating $H^1_{\cF^\bullet(n)}(K_w,\sT)$ to be the image (resp.~preimage) of $H^1_{\cF^\bullet(n)}(K_w,\TT/I_n\TT)$.
\end{de}
\begin{rem}\label{rem:selmer-comparison}
    We note that $H^1_{\bullet}(K_w,T) = H^1_\f(K_w,T)$ for $w\mid p$, so $\Sel_\bullet(K,T)$ coincides with the usual Bloch--Kato Selmer group $\Sel(K,T)$. This is due to the colinearity (the deduction is similar to \cite[Remark 5.8]{LLZ11})
    \[
        (p-1)\col^\sharp + (a_p-2)\col^\flat \equiv 0\bmod X,
    \]
    so $H^1_\sharp(K_w,T) = H^1_\flat(K_w,T) = \ker(\cL_w \bmod X)$. Now, as the Perrin-Riou regulator $\cL_w$ specializes to the dual exponential modulo $X$
    \[
      (1-\varphi)\left(1-p^{-1}\varphi^{-1}\right)^{-1}\circ  \exp^* \colon H^1(K_w,T) \to \Dcris(V),
    \]
    and $(1-\varphi)\left(1-p^{-1}\varphi^{-1}\right)^{-1}$ is invertible on $\Dcris(V)$,
    we see that the kernel is exactly $H^1_\f(K_w,T)$ by \cite[Definition 3.10, Corollary 3.8.4]{bloch-kato:l-fcts-tamagawa-numbers}.
\end{rem}

\begin{lem}\label{lem:p-power-kolyvagin}
There exists an integer $d$ that is independent of $n$ such that $$\varpi^d\bka_n^\bullet\in H^1_{\cF^\bullet(n)}(K,\TT/I_n\TT)$$
    for both $\bullet\in\{\sharp,\flat\}$.
\end{lem}
\begin{proof}
    This follows from the same proof as \cite[Lemma~A.2]{castella-wan24}.
\end{proof}

Next, recall from \cite[\S1.2]{howard04} that given a $p$-adic coefficient ring $\sR$, a $\sR$-linear $G_K$-representation $\sT$ finite free over $\sR$, a Selmer condition $\sF$ on $\sT$, and a subset $\sL\subset \cL$ containing $\cL^{(e)}=\{\ell\in \cL\colon a_\ell\equiv \ell+1\equiv 0\bmod \varpi^e\}$ for some $e\ge 1$, a Kolyvagin system for the triple $(\sT,\sF,\sL)$ is a set of classes $\{\kappa_n\}_n$, where $n$ ranges over
\[
    \cN(\sL) = \{\text{squarefree product of primes in }\sL\},
\]
such that
\begin{enumerate}
    \item[(i)] $\kappa_n\in H^1_{\sF(n)}(K,T/I_nT)$;
    \item[(ii)] for all $n\ell\in \cN(\sL)$,
    \[
        \loc_\ell(\kappa_{n\ell}) = \phi^{\rm fs}_\ell \circ\loc_\ell(\kappa_n).
    \]
    Here, we have identified $H^1_{\tr}(K_\ell,M)$ with $H^1_{\rm s}(K_\ell,M)$  by Proposition 1.1.9, \textit{ibid.}.
\end{enumerate}
\begin{prop}\label{prop:kappa-kolyvagin}
    There exists $d\in \Z_{\ge 0}$ such that both collections of cohomology classes $\{\varpi^d\bka^\bullet_n\}_{n\in \cN}$ with $\bullet\in \{\sharp,\flat\}$ are Kolyvagin systems for $(\TT,\cF_\bullet,\cL)$.
\end{prop}
\begin{proof}
    Thanks to Lemma \ref{lem:p-power-kolyvagin}, it suffices to verify the compatibility
    \[
        \loc_\ell(\varpi^d\bka^\bullet_{n\ell}) = \phi^{\rm fs}_\ell \circ\loc_\ell(\varpi^d\bka^\bullet_n).
    \]
    This can be shown in the way similar to \cite[Proposition 1.7.4]{howard04}; note that the relation 
    \[
    P[\ell n p^k]\equiv \operatorname{Frob}(\lambda_{ np^k}) P[np^k]\mod \lambda_{\ell n p^k},
    \]
    where $\lambda_\star$ denotes a prime of $K[\star]$ above $\ell$,
    used therein is a consequence of \eqref{eq:norm-l} and the Eichler--Shimura congruence relation \cite[\S8.7]{diamond-shurman}.
\end{proof}
We have the local Tate pairing
\[
H^1(K_w,\TT)\times \prod_{\lambda|w}H^1(K_{\infty,\lambda},A)\to F/\cO.
\]

Let $H^1_\bullet(K_{\infty,\lambda},A)\subset H^1(K_{\infty,\lambda},A)$ so that $\prod_{\lambda|w}H^1_\bullet(K_{\infty,\lambda},A)$ is the orthogonal complement of $H^1_\bullet(K_w,\TT)$ under the local Tate pairing.
We define the discrete and compact signed Selmer groups as follows.

\begin{de}
For $\bullet\in\{\sharp,\flat\}$, we define the discrete signed Selmer group $\Sel_\bullet(K_\infty,A)$ as the kernel of the localization map
\[
H^1(K_\infty,A)\to\prod_{\lambda\nmid p}\frac{H^1(K_{\infty,\lambda},A)}{H^1_\mathbf{f}(K_{\infty,\lambda},A)}\times\prod_{\lambda|p} \frac{H^1(K_{\infty,\lambda},A)}{H^1_\bullet(K_{\infty,\lambda},A)},
\]
and the compact Selmer group $\Sel_\bullet(K,\TT)$ as the kernel of the localization map
\[
H^1(K,\TT)\to\prod_{w\nmid p}\frac{H^1(K_{w},\TT)}{H^1_\mathbf{f}(K_{w},\TT)}\times\prod_{w|p} \frac{H^1(K_{w},\TT)}{H^1_\bullet(K_w,\TT)}.
\]

Given a character $\alpha:\Gamma\to\cO^\times$ and a prime $w$ of $K$ lying above $p$, let $H^1_\bullet(K_w,T_{\alpha^{-1}})\subset H^1(K_w,T_{\alpha^{-1}})$ denote the image of $H^1_\bullet(K_w,\TT)$ under the specialization map $H^1(K_w,\TT)\to H^1(K_w,T_{\alpha^{-1}})$. We then define the signed Selmer group of $T_{\alpha^{-1}}$ over $K$, denoted by $\Sel_\bullet(K,T_{\alpha^{-1}})$, as the kernel of 
\[
H^1(K,T_{\alpha^{-1}})\to\prod_{w\nmid p}\frac{H^1(K_{w},T_{\alpha^{-1}})}{H^1_\mathbf{f}(K_{w},T_{\alpha^{-1}})}\times\prod_{w|p} \frac{H^1(K_{w},T_{\alpha^{-1}})}{H^1_\bullet(K_w,T_{\alpha^{-1}})}.
\]

Let $H^1_\bullet(K_w,A_\alpha)\subset H^1(K_w,A_\alpha)$ be the orthogonal complement of $H^1_\bullet(K_w,T_{\alpha^{-1}})$ under the local Tate pairing
\[
H^1(K_w,T_{\alpha^{-1}})\times H^1(K_w,A_{\alpha})\to F/\cO.
\]
We define the signed Selmer group of $A_\alpha$ over $K$, denoted by $\Sel_\bullet(K,A_\alpha)$, as the kernel of 
\[
H^1(K,A_{\alpha})\to\prod_{w\nmid p}\frac{H^1(K_{w},A_{\alpha})}{H^1_\mathbf{f}(K_{w},A_{\alpha})}\times\prod_{w|p} \frac{H^1(K_{w},A_{\alpha})}{H^1_\bullet(K_w,A_\alpha)}.
\]
\end{de}

We recall the following signed Iwasawa main conjectures that were studied in \cite{BBL2,BLV}.
\begin{conj}\label{conj:main}
    For $\bullet\in\{\sharp,\flat\}$, the $\Lambda$-module $\Sel_\bullet(K_\infty,A)^\vee$ is of rank one. Furthermore, 
    \[
        \Char_{\Lambda}\left(\Sel_\bullet(K,\TT)/\Lambda\cdot \bka_1^\bullet\right)^2 \subseteq \Char_{\Lambda}\left(\Sel_\bullet(K_\infty,A)^\vee_\mathrm{tor}\right)
    \]
    as ideals of $\Lambda\otimes\QQ_p$. Moreover, if $p$ is either split in $K$ or $\bullet = \flat$, the inclusion above is an equality.
\end{conj}

Finally, we define the signed Shafarevich--Tate groups, which will play a crucial role in subsequent sections.
\begin{de}
For $\bullet\in\{\sharp,\flat\}$ and a character $\alpha:\Gamma\to\cO^\times$, we define the signed Shafarevich--Tate group of $A_\alpha$ as
\[
\Sha_\bullet(K,A_\alpha)=\Sel_\bullet(K,A_\alpha)/\Sel_\bullet(K,A_\alpha)_\div,
\]
where $\Sel_\bullet(K,A_\alpha)_\div$ denotes the maximal divisible subgroup of $\Sel_\bullet(K,A_\alpha)$.
\end{de}
Note that $\Sha_\bullet(K,A_\alpha)$ has finite cardinality by definition.

\section{Calculating Euler characteristics for signed Selmer groups}
	The main goal of this section is to calculate the Euler characteristics of Selmer modules $\Sel_{\bullet}(K_\infty,A)$ for $\bullet\in\{\sharp,\flat\}$, which will be used in the proof of Theorem~\ref{thmA}. Since they are not expected to be torsion in light of Conjecture~\ref{conj:main}, this renders the standard approach of Greenberg \cite[\S4]{greenberg-cetraro} inapplicable. We develop a novel method that works for Selmer modules of arbitrary $\Lambda$-coranks, at the price of working with various character twists. As in the previous sections, $V$ is the 2-dimensional Galois representation attached to $f$, $T$ a fixed Galois-stable lattice, and $A = V/T$. Throughout, we fix a non-trivial character $\alpha: \Gamma\to \cO^\times$. Set $m\in \ZZ_{\ge 0}$ to be the largest integer such that $\alpha\equiv 1\bmod \varpi^m$.

	The main result of this section is the following theorem.
	\begin{thm}\label{thm:euler-char}
		Let $\bullet\in\{\sharp,\flat\}$, and let $\cF_\bullet(X)$ be a generator of $\Char((\Sel_\bullet(K_\infty,A)^\vee)_{\tors})$, where $(\Sel_\bullet(K_\infty,A)^\vee)_\tors$ denotes the maximal $\Lambda$-torsion submodule of $\Sel_\bullet(K_\infty,A)^\vee$. For all but finitely many $\alpha$, we have
		\[
			|\cO/\cF_\bullet(\alpha(\gamma)-1)| \le |\Sha_\bullet(K,A_\alpha)| \prod_{w\mid N^+}c_w^{(p)}.
		\]
		Here, $c^{(p)}_w = \#H^1_\ur(K_w,A)<\infty$ is the $p$-part of the local Tamagawa number.
	\end{thm}

    The proof of Theorem~\ref{thm:euler-char} will occupy the remainder of this section and will be divided into several steps.

\subsection{The groundwork}

    To simplify the notation, we write $\cX = (\Sel_\bullet(K_\infty,A)^\vee)_\tors$, and let $\fP_\alpha = (\gamma - \alpha^{-1}(\gamma))\subset \Lambda$ denote the height-one prime attached to $\alpha^{-1}$ (as before, $\gamma$ is a topological generator of $\Gamma$). Note that $\TT/\fP_\alpha\TT = T_\alpha$. We begin with the following preliminary lemma.
	\begin{lem}\label{lem:euler-char}
		Let $\cF_\bullet(X)$ be a generator of $\Char((\Sel_\bullet(K_\infty,A)^\vee)_{\tors})$. For all but finitely many $\alpha$, both $\cX/\fP_{\alpha^{-1}}$ and $\cX[\fP_{\alpha^{-1}}]$ are finite, and
		\begin{align}\label{eq:leadingterm}
		\left|	\cO/\cF_\bullet(\alpha(\gamma)-1)\right| = |\cX/\fP_{\alpha^{-1}}|/|\cX[\fP_{\alpha^{-1}}]|\le |\cX/\fP_{\alpha^{-1}}|.
		\end{align}
	\end{lem}
	\begin{proof}
		Note that by excluding finitely many $\alpha$, we can ensure that the characteristic polynomial of $\cX$ does not vanish at $\alpha(\gamma)$, so $\cX/\fP_{\alpha^{-1}}$ is finite. Then the standard argument from \cite[Lemma 4.2]{greenberg-cetraro} applies, and gives the finiteness of $\cX[\fP_{\alpha^{-1}}]$, as well as \eqref{eq:leadingterm}.
	\end{proof}
	
	Our next step is to relate $\cX/\fP_{\alpha^{-1}}$ to $\Sel_\bullet(K_\infty,A)^\vee/\fP_{{\alpha}^{-1}}$. For this, we shall recall the control theorem for signed Selmer groups. Below, let 
\begin{equation}
\label{eq:v_alpha}                  v_\alpha:\Sel_\bullet(K,A_\alpha)\to\Sel_\bullet(K_\infty,A)[\fP_\alpha]
\end{equation}
 be the natural map induced by the restriction on the Galois cohomology groups.
	\begin{prop}\label{prop:control}
		When $\alpha$ is nontrivial and $m\gg 0$, the map $v_\alpha$ has finite kernel and cokernel. In particular, for all but finitely many $\alpha$,
        \[ \corank_{\Lambda}\Sel_\bullet(K_\infty,A) = \corank_{\cO}\Sel_\bullet(K,A_\alpha).
        \]
	\end{prop}
	\begin{proof}
		The argument of \cite[Lemma 2.3]{ponsinet20} applies verbatim here, as the assumption (Tors.) \textit{op.~cit.}, which asserts that $H^0(K_w,A)=0$ for $w|p$, holds under our ongoing hypotheses by Lemma \ref{lem:T-no-invariant}. The ``in particular'' part can be shown via an application of the structure theorem of finitely generated $\Lambda$-modules.
	\end{proof}

    \subsection{Measuring $|\cX/\fP_{\alpha^{-1}}|$}
    
	Denote by $R\in \Z_{\ge 0}$ the $\Lambda$-corank of $\Sel_\bullet(K_\infty,A)$. By the structure theorem of finitely generated $\Lambda$-modules, there exists a pseudo-null module $Q$ and a short exact sequence
	\[
		0\to \Sel_\bullet(K_\infty,A)^\vee/\tors \to \Lambda^{\oplus R} \to Q \to 0.
	\]
While Conjecture \ref{conj:main} predicts that $R=1$, we do not need the precise value of $R$ in our proof of Theorem \ref{thm:euler-char} below.
	\begin{prop}\label{prop:sel-sha}
		Suppose that $m\gg 0$ and $\alpha$ is nontrivial. Then 
		\[
			\frac{|\cX/\fP_{\alpha^{-1}}|}{|\Sha_\bullet(K,A_\alpha)|} =  \frac{|\Sel_\bullet(K_\infty,A)[\fP_\alpha]/\im(\Sel_\bullet(K,A_\alpha))|}{|Q^\Gamma|},
		\]
		where the image of $\Sel_\bullet(K,A_\alpha)$ is under the restriction map $v_\alpha$ in \eqref{eq:v_alpha}.
	\end{prop}
	\begin{proof}
		Let $M_\infty\subseteq \Sel_\bullet(K_\infty,A)$ such that $M_\infty^\vee = \Sel_\bullet(K_\infty,A)^\vee/(\Sel_\bullet(K_\infty,A)^\vee)_{\tors}$. As we have an injection $\Sel_\bullet(K_\infty,A)[\fP_{\alpha}]/M_\infty[\fP_{\alpha}] \hookrightarrow (\cX/\fP_{\alpha^{-1}})^\vee$, Lemma \ref{lem:euler-char} shows that $M_\infty[\fP_{\alpha}]$ contains the maximal divisible subgroup of $\Sel_\bullet(K_\infty,A)[\fP_{\alpha}]$. Moreover, taking the right-derived functor of $\fP_{\alpha^{-1}}$-torsion on the exact sequence
		\[
		0\to M_\infty^\vee \to \Lambda^{\oplus R}\to Q\to 0
		\]
		gives the following exact sequence when $m\gg 0$:
		\[			0\to Q^{\Gamma} \to M_\infty^\vee/\fP_{\alpha^{-1}} \to \cO^{\oplus R} \to Q_\Gamma\to 0.
		\]
		It follows that $M^\vee_\infty/\fP_{\alpha^{-1}} \simeq \cO^{\oplus R}\oplus Q^\Gamma$, and thus $M_\infty[\fP_{\alpha}] \simeq (F/\cO)^R\oplus (Q^\Gamma)^\vee$.
		
		Next, consider the following commutative diagram with exact rows
		\[
		\begin{tikzcd}
			0 \ar[r] & \Sel_\bullet(K,A_\alpha)_\div \ar[r] \ar[d,"u"]&
			\Sel_\bullet(K,A_\alpha) \ar[r] \ar[d,"v_\alpha"]&
			\Sha_\bullet(K,A_\alpha) \ar[r] \ar[d,"w"] &
			0\\
			0 \ar[r] & M_\infty[\fP_{\alpha}] \ar[r] &
			\Sel_\bullet(K_\infty,A)[\fP_{\alpha}] \ar[r] &
			\Sel_\bullet(K_\infty,A)_\cotors[\fP_{\alpha}] \ar[r] &
			0.
		\end{tikzcd}
		\]
		(The surjectivity of the second row is due to $M_\infty/\fP_{\alpha} = (M_\infty^\vee[\fP_{\alpha^{-1}}])^\vee =0$.)
		
		By the control theorem (Proposition \ref{prop:control}) and the finiteness of $\Sha_\bullet(K,A_\alpha)$, we see that $\ker(u)$ and $\coker(u)$ are finite. Dualizing, the map $M^\vee_\infty/\fP_{\alpha^{-1}} \xrightarrow{u^\vee} (\Sel_\bullet(K,A_\alpha)_\div)^\vee$ has finite kernel, and its target is a free $\cO$-module, so its kernel must be the maximal torsion submodule of $M^\vee_\infty/\fP_{\alpha^{-1}}$, namely $Q^\Gamma$. Thus $\coker(u) \simeq (Q^\Gamma)^\vee$.
		
		Note that $v_\alpha$ is induced from the restriction map $H^1(K,A_\alpha)\to H^1(K_\infty,A) = H^1(K_\infty,A_\alpha)$. By the inflation-restriction sequence, the kernel of the latter map is given by $H^1(K_\infty/K,A_\alpha^{G_{K_\infty}})$. Recall that $A_\alpha^{G_{K_\infty}} = A^{G_{K_\infty}} = 0$ by Lemma \ref{lem:T-no-invariant}. Thus, $\ker(v_\alpha)=0$. The snake lemma then implies $\ker(u) = 0$, and we deduce that
		\begin{align*}
			\frac{|\Sel_\bullet(K_\infty,A)_\cotors[\fP_{\alpha}]|}{|\Sha_\bullet(K,A_\alpha)|} &= \frac{|\coker(w)|}{|\ker(w)|}\\
            &= \frac{|\coker(v_\alpha)||\ker(u)|}{|\ker(v_\alpha)||\coker(u)|}\\
            &= \frac{|\Sel_\bullet(K_\infty,A)[\fP_\alpha]/\im(\Sel_\bullet(K,A_\alpha))|}{|Q^\Gamma|},
		\end{align*}
        as desired.
	\end{proof}
	
	Our next goal is to bound the quantity $|\Sel_\bullet(K_\infty,A)[\fP_\alpha]/\im(\Sel_\bullet(K,A_\alpha))| = |\coker(v_\alpha)|$ by a product of certain explicit local constants. Denote by $\Sigma$ the subset of places in $K$ given by $\{w\colon w\mid Np\}$, and let $K_\Sigma/K$ be the maximal extension unramified outside $\Sigma\cup\{\infty\}$. For $L\in \{K,K_\infty\}$, put
	\begin{align}
		\cP_\bullet^\Sigma(L,A_\alpha) = \prod_{\substack{w:\text{ place of }L\\ w\mid Np}} H^1_{/\bullet}(L_w,A_\alpha),
	\end{align}
	where
	\begin{align*}
		H^1_{/\bullet}(L_w,A_\alpha) = \begin{cases}
			H^1(L_w,A_\alpha)/H^1_\f(L_w,A_\alpha) & \text{if }w\mid N;\\
			H^1(L_w,A_\alpha)/H^1_\bullet(L_w,A_\alpha) & \text{if }w\mid p.
		\end{cases}
	\end{align*}
	Furthermore, denote
	\[		\cG^\Sigma_\bullet(L,A_\alpha) = \im\left(H^1(K_\Sigma/L,A_\alpha)\to \cP^\Sigma_\bullet(L,A_\alpha)\right).
	\]
	As in \cite[\S3]{greenberg-cetraro}, consider the diagram
	\[
	\begin{tikzcd}
		0 \ar[r] & \Sel_\bullet(K,A_\alpha) \ar[d,"v_\alpha"] \ar[r]  & H^1(K_\Sigma/K,A_\alpha) \ar[r] \ar[d,"h"] &  \cG^\Sigma_\bullet(K,A_\alpha) \ar[r] \ar[d,"g"] & 0\\
		0 \ar[r] & \Sel_\bullet(K_\infty,A)[\fP_\alpha] \ar[r] & H^1(K_\Sigma/K_\infty,A)[\fP_\alpha] \ar[r] & \cP^\Sigma_\bullet(K_\infty,A)[\fP_\alpha]. &
	\end{tikzcd}
	\]
	Since $A_\alpha^{G_{K_\infty}} = 0$, we see that $h$ is an isomorphism, and hence $\coker(v_\alpha)\simeq \ker(g)$ by the snake lemma. In what follows, for $w\in \Sigma$, let $r_w$ denote the local restriction map
		\begin{align*}
			r_w\colon H^1_{/\bullet}(K_w,A_\alpha)\to \prod_{w'\mid w} H^1_{/\bullet}(K_{\infty,w'},A).
		\end{align*}
    Let $r: \cP^\Sigma_\bullet(K,A_\alpha) \to \cP^\Sigma_\bullet(K_\infty,A)$ denote the product of the maps $r_w$'s. By definition, we have
    \[
        \ker(g)\subseteq \ker(r) = \prod_{w\mid Np}\ker(r_w),
    \]
    and below we shall establish the finiteness of
	\[
		\ker(r_w) =\ker\left(H^1_{/\bullet}(K_w,A_\alpha) \to \prod_{w'} H^1_{/\bullet}(K_{\infty,w'},A)\right)
	\]
	for all $w\mid Np$. 
    
    \subsection{Local Tamagawa numbers and concluding the proof}
    
    First note that for any place $w'$ of $K_{\infty}$ above $w$, the kernel of $r_{w'}: H^1_{/\bullet}(K_w,A_\alpha)\to H^1_{/\bullet}(K_{\infty,w'},A)$ does not depend on the choice of $w'$, since the decomposition groups are Galois conjugates of each other. In particular, $\ker(r_w) = \ker(r_{w'})$ for any choice of $w'$.
	\begin{prop}\label{prop:cp}
		Let $w\in \Sigma$ and $w'\mid w$ be a place of $K_\infty$. We have $\ker(r_{w'}) = 0$ unless $w\mid N^+$. In the latter case, $\ker(r_{w'})$ is finite for $m\gg 0$, and has cardinality
		\[
			c_w^{(p)} = \left|\frac{A^{G_{K_{\infty,w'}}}/\div}{(\gamma-1)\left(A^{G_{K_{\infty,w'}}}/\div\right)}\right|.
		\]
	\end{prop}
	\begin{proof}
		Suppose first that $w\mid p$. Since $H^1_\bullet(K_w,T_\alpha)$ is defined as the image of $H^1_\bullet(K_w,\TT)$, we deduce that $r_w$ is injective after dualizing. 
        
        For the rest of the proof, we assume $w\mid N$. We recall from \cite[proof of Lemma 4.4]{lei-mastella-zhao}, which asserts that there is a short exact sequence
		\begin{align}\label{eq:local-tama}
			0\to H^1_\f(K_w,A_\alpha) \to\ker(H^1(K_w,A_\alpha)\to H^1(K_{\infty,w'},A)) \to \ker(r_{w'})\to 0.
		\end{align}
		We will show that the middle term is finite. This then implies that the first term is zero, as $H^1_\f(K_w,A_\alpha)$ is by definition the image of $H^1_\f(K_w,V_\alpha)$ and therefore is divisible. It then follows that the cardinality of $\ker(r_w')$ is the same as that of the middle term.
		
		To study $\ker(H^1(K_w,A_\alpha)\to H^1(K_{\infty,w'},A))$ for $w| N$, there is a dichotomy: Either $w$ splits completely in $K_\infty$, which happens when the rational prime $\ell$ below $w$ is inert in $K$ \cite[Exercise 1.10]{greenberg-park-city}; or $w$ decomposes into finitely many primes in $K_\infty$, which occurs when the rational prime below $w$ is split \cite[Theorem 2]{brink}. As $w| N$, the first case applies to those $w| N^-$ and the second to $w| N^+$. 
		
		In the first scenario, we have $K_{\infty,w'} = K_w$, so $\ker(r_{w'}) = 0$ plainly by \eqref{eq:local-tama}. In the second, denote $B_w = A^{G_{K_{\infty,w'}}}$, which is cofinitely generated over $\cO$. The inflation-restriction exact sequence tells us that
		\[
			\ker(H^1(K_w,A_\alpha)\to H^1(K_{\infty,w'},A))\simeq H^1(K_{\infty,w'}/K_w,B_w(\alpha))\simeq 
			B_w(\alpha)/(\gamma_{w'}-1)B_w(\alpha),
		\]
		where $\gamma_{w'}$ is any topological generator of $\Gal(K_{\infty,w'}/K_w)\simeq \Z_p$; we may identify the last term with $[B_w/(\alpha(\gamma)\gamma -1)B_w](\alpha)$. By duality, we have
		\[
			\#(B_w/(\alpha(\gamma)\gamma -1)B_w) = \#(B_w^\vee)^{\gamma=\alpha(\gamma)}.
		\]
		Since $B_w^\vee\otimes \Q_p$ is a finite-dimensional vector space, the action of $\gamma\in\Gamma$ on this space admits finitely many eigenvalues. Moreover, by \cite[proof of Lemma 4.3]{lei-mastella-zhao}, $(B_w^\vee\otimes \Q_p)^{\gamma=1}=0$. Therefore, when $m\gg0$,
        \[
			(B_w^\vee)^{\gamma=\alpha(\gamma)} = B_w^\vee[\varpi^\infty]^{\gamma=\alpha(\gamma)} = B_w^\vee[\varpi^\infty]^{\gamma=1},
		\]
		where the last equality is due to $B_w^\vee[\varpi^\infty]$ being finite. The desired finiteness of 			$\ker(H^1(K_w,A_\alpha)\to H^1(K_{\infty,w'},A))$ follows after taking duals.
	\end{proof}
	The constant $c_w^{(p)}$ in fact coincides with the $p$-part of the local Tamagawa number:
	\begin{cor}\label{cor:tama}
		For $w\mid N^+$, the constant $c_w^{(p)}$ above equals $\#H^1_\ur(K_w,A)$.
	\end{cor}
	\begin{proof}
		By taking $\alpha = \1$, the trivial character, the proof of Proposition \ref{prop:cp} shows that
		\[
			c_w^{(p)} = \#\left(A^{G_{K_{\infty,w'}}}/(\gamma-1)A^{G_{K_{\infty,w'}}}\right) = \#\ker(H^1(K_w,A)\to H^1(K_{\infty,w'},A)),
		\]
		thanks to \cite[Lemma 4.3]{lei-mastella-zhao}. Note that $K_{\infty,w'}$ is the unique unramified $\ZZ_p$-extension of $K_w$ by \cite[Proposition 13.2]{washington:book}. Denoting by $I_w$ the inertia group of $G_{K_w}$, we see that $G_{K_{\infty,w'}}/I_w$ is the prime-to-$p$ part of $G_{K_w}/I_w\simeq \hat{\ZZ}$. As such, $H^1(G_{K_{\infty,w'}}/I_w,A^{G_{K_{\infty,w'}}})=0$ as $A$ is $p$-primary. Consequently, the restriction $H^1(K_{\infty,w'},A)\to H^1(I_w,A)$ is injective, whereby
		\[
        \#\ker(H^1(K_w,A)\to H^1(K_{\infty,w'},A)) = \#\ker(H^1(K_w,A)\to H^1(I_w,A)) = \#H^1_\ur(K_w,A),
		\]
        as desired.
	\end{proof}

    \begin{proof}[Proof of Theorem \ref{thm:euler-char}]
        In the discussion preceding Proposition \ref{prop:cp} we have seen that
        \[
            |\Sel_\bullet(K_\infty,A)[\fP_\alpha]/\im(\Sel_\bullet(K,A_\alpha))|\le \prod_{w\mid Np} |\ker(r_w)| = \prod_{w\mid N^+} c_w^{(p)},
        \]
        where the last equality is by Proposition \ref{prop:cp} and Corollary \ref{cor:tama}. Using Proposition \ref{prop:sel-sha}, we find
        \[
            |\cX/\fP_{\alpha^{-1}}| \le |\Sha_\bullet(K,A_\alpha)| \prod_{w\mid N^+} c_w^{(p)}.
        \]
        The proof is then concluded by invoking Lemma \ref{lem:euler-char}.
    \end{proof}

\section{Bounding signed Selmer groups}
The main goal of this section is to show how to apply Howard's results on Kolyvagin systems established in \cite{howard04-gl2} to bound the signed Selmer groups.

Let $\bullet \in \{\sharp,\flat\}$. Throughout this section, we assume the following hypothesis holds:
\begin{align}\tag{nontors.}\label{hyp:nontors}
    \bka_1^\bullet\in H^1(K,\TT) \text{ is not $\Lambda$-torsion.}
\end{align}

Thus, assuming \eqref{hyp:nontors}, for a fixed character $\alpha:\Gamma\to\cO^\times$  such that $\alpha\equiv1\mod \varpi^m$, up to finitely many exceptions, $\bka_1^\bullet(\alpha)$ is non-torsion over $\cO$. 
Write $\FF = \cO/\varpi$ and $\bar T = T/\varpi T$. Denote by $\rho_f:G_K\to \Aut_\cO(T)\simeq \GL_2(\cO)$ the Galois representation attached to $f$, and $\bar \rho_f: G_K\to \Aut_\FF(\bar T) \simeq \GL_2(\FF)$ its reduction modulo $\varpi$. Let $K_T/K$ be the Galois extension cut out by $\rho_f$; note that by the existence of the Weil pairing, $K_T(\mu_{p^\infty}) = K_T$.

\begin{lem}\label{lem:central}
    The image of $\rho_f$ contains the set of scalar matrices $\{\lambda I_2:\lambda\in\mu_{p-1}\}$.
\end{lem}
\begin{proof}
    Identifying $\mu_{p-1}$ with the set $\{\lambda I_2:\lambda\in\mu_{p-1}\}\subset \GL_2(\cO)$ and also its image in $\GL_2(\FF)$, we first show $\im(\bar\rho_f)\supset \mu_{p-1}$. Since $p$ is unramified in $K$, by \cite[Theorem~2.6]{edixhoven92}, the inertia group at a prime above $p$ acts as
    \[
    \begin{bmatrix}
        \omega_2&0\\0&\omega_2^{p}
    \end{bmatrix},
    \]
    where $\omega_2$ is a fundamental character of level $2$, and is therefore of order $p^2-1$. As
     \[
    \begin{bmatrix}
        \omega_2&0\\0&\omega_2^{p}
    \end{bmatrix}^{p+1}=\begin{bmatrix}
        \omega&0\\0&\omega
    \end{bmatrix},
    \]
   where $\omega$ is a fundamental character of level one, and is of order $p-1$. This shows that $\im(\bar\rho_f)$ contains $\mu_{p-1}$.
		
		Let $\nu\in \mu_{p-1}$, and take $g\in \im(\rho_f)$ to be a lift of $\nu$. Then, as an element in $\GL_2(\cO)$, we can write
		\[
			g = \nu(1+\varpi x) \quad \text{ for some }x\in M_{2\times 2}(\cO).
		\]
		Since $\nu$ is in the center of $\GL_2(\cO)$, it follows that
		\[
			\lim_{n\to \infty} g^{p^n}= \lim_{n\to \infty} \nu(1+\varpi x)^{p^n} = \nu,
		\]
		which belongs to the closed subgroup $\im(\rho_f)\subseteq \GL_2(\cO)$, as desired. 
\end{proof}

\begin{cor}\label{cor:hypothesis-1}
		We have $H^1(K_T/K,\bar T) = 0$.
	\end{cor}
	\begin{proof}
		By definition, $H^1(K_T/K,\bar T) = H^1(\im(\rho_f),\FF^2)$. Lemma \ref{lem:central} tells us that $\im(\rho_f)$ contains the scalar subgroup $\mu_{p-1}$. The inflation-restriction gives an exact sequence
		\[
			H^1(\im(\rho_f)/\mu_{p-1}, (\FF^2)^{\mu_{p-1}})\to H^1(\im(\rho_f),\FF^2) \to H^1(\mu_{p-1},\FF^2).
		\]
		The first entry is zero as $(\FF^2)^{\mu_{p-1}} = 0$, so is the last since $\FF^2$ is a $p$-group. This yields the desired vanishing.
	\end{proof}

From now on, the following hypothesis on $\bar T$ will be in effect:
\begin{align}\tag{irr.}\label{hyp:irr}
    \text{If $p$ is inert in $K$, the residual $G_K$-representation }\bar T\text{ is absolutely irreducible.}
\end{align}

\begin{rem}\label{rem:irreducibility}When $p$ splits in $K$, the irreducibility of $\bar T$ as a $G_{\QQ_p}$-representation is a classical result of Fontaine (see \cite[Theorem 2.6]{edixhoven92}). In particular, $\bar T$ is also absolutely irreducible as a $G_K$-representation. 

If $p$ is inert, the absolute irreducibility of $\bar T$ might not hold, for example, when $f$ is a CM form. If $f$ is not a CM form, it follows from \cite[Theorem~2.1]{ribet:rankin} that the hypothesis \eqref{hyp:irr} is valid for almost all $p$. If $\cA_f$ is a semi-stable elliptic curve defined over $\QQ$, the image of the $G_\QQ$-representation $\cA_f[p]$ is $\GL_2(\FF_p)$ by \cite[Proposition 2.1]{edixhoven97}. As $G_K$ is a subgroup of $G_\QQ$ of index $2$, the hypothesis \eqref{hyp:irr} is valid in this case. 
\end{rem}
Given a subset $\cL'\subset\cL$, we let $\cN(\cL')$ be the set of squarefree products of elements of $\cL'$.
For a positive integer $e$, put
\[
\cL^{(e)}=\{\ell\in\cL_0:a_\ell\equiv \ell+1\equiv 0\mod \varpi^e\}.
\]
The following theorem is a generalization of \cite[Theorem~1.3.1]{bcgs}.
\begin{thm}\label{thm:howard}
    Assume \eqref{hyp:irr}, and let $\alpha: \Gamma\to \cO^\times$ be a character. Let $\cL'\subset \cL$ such that $\cL^{(e)}\subset \cL'$ for $e\gg0$. Let $\cN' = \cN(\cL')$. Suppose that there is a collection of cohomology classes
    \[
    \{\kappa_n\in H^1(K,T_\alpha/I_nT_\alpha) :n\in\cN'\}
    \]
    with $\kappa_1\ne 0$ and that there exists an integer $t\ge0$, independent of $n$, such that $\{\varpi^t\kappa_n\}$ is a Kolyvagin system for $(T_\alpha,\cF_\bullet,\cL')$. Then, 
    $\Sel_\bullet(K,T_\alpha)$ is a free $\cO$-module of rank one, and there is a finite $\cO$-module $M$ such that
    \[
    \Sel_\bullet(K,A_{\alpha})\cong F/\cO\oplus M\oplus M
    \]
    with $\len_{\cO}(M)\le\ind(\kappa_1)$, where $\ind(\kappa_1)=\len_{\cO}(\Sel_\bullet(K,T_{\alpha})/\cO\cdot \kappa_1)$.
\end{thm}
\begin{proof}

    We will verify that the representation $T_\alpha$ and both the $\sharp$- and $\flat$-Selmer structures satisfy the hypotheses (H1-5) in \cite[\S2.2]{howard04-gl2}. Once this is done, the result follows from the general machinery of Theorem 1.6.1 \textit{op.~cit.}, and its modification \cite[Theorem 2.2.2]{howard04-gl2} when $t>0$. 

    Below, we write $\bar M = M/\varpi M$ for $M= T,T_\alpha$. Firstly, Corollary \ref{cor:hypothesis-1} shows that
    \[
        H^1(K_T(\mu_{p^\infty})/K,\bar T)=H^1(K_T(\mu_{p^\infty})/K,\bar T_\alpha) = 0,
    \]
    which means that (H1) holds. Concerning (H2), we have $\bar T_\alpha = \bar T$, so the absolute irreducibility is ensured by \eqref{hyp:irr}, and $\bar T$ already comes with a $G_{\Q}$-action. The existence of the equivariant pairing in (H3), to be established below, then ensures the plus/minus $\tau$-eigenspace to be one-dimensional.

    We now turn our attention to (H3), which states that there is a perfect symmetric $\cO$-linear pairing
    \[
    (\ ,\ ):T_\alpha\times T_\alpha\to \cO(1)
    \]
    such that $(a^\sigma,b^{\tau\sigma\tau})=(a,b)^\sigma$ for $a,b\in T$ and $\sigma\in G_K$, where $\tau$ is the complex conjugation. Further, the induced pairing on $\bar T_\alpha\times \bar T_\alpha$ satisfies $(a^\tau,b^\tau)=(a,b)^\tau$.
    The Weil pairing induces a $G_K$-equivariant perfect pairing $\langle\ ,\ \rangle$ on $T_\alpha\times T_{\alpha^{-1}}$. Since $\alpha$ is an anticyclotomic character, we have an isomorphism $\jmath: T_{\alpha}\to T_{\alpha^{-1}}$ sending $y\otimes t\in T\otimes \Z_p(\alpha)$ to $y^\tau\otimes t\in T\otimes\Z_p(\alpha^{-1})$, such that $\jmath(x^\sigma) = \jmath(x)^{\tau\sigma\tau}$ for $\sigma \in G_K$. We can then define $(x,y)$ to be $\langle x,\jmath(y)\rangle$ for $x,y\in T_\alpha$. The desired properties of the pairing can then be checked using the fact that $\tau$ acts on $\cO(1)$ by $-1$.

    For (H4), we need to verify that under $(\ ,\ )$, the local condition $H^1_\bullet(K_v,T_\alpha)$ is orthogonal complement of $H^1_\bullet(K_v,T_{\alpha^{-1}})$, which follows from the discussions in \S\ref{S:orth-split} and \S\ref{S:orth-inert}. Finally, (H5) requires that $\bigoplus_{w|v}H^1_\bullet(K_w,\bar T_\alpha)$ is stable under the action of $\Gal(K/\QQ)$. This holds since the complex conjugation sends $H^1_\bullet(K_w,\bar T_\alpha)$ to $H^1_\bullet(K_{\bar w},\bar T_\alpha)$, and vice versa.
\end{proof}

\begin{cor}\label{cor:rank-sel}
		Assume \eqref{hyp:irr} and \eqref{hyp:nontors} hold. Then, we have the equalities $$\rank_\Lambda(\Sel_\bullet(K,\TT)) = \corank_\Lambda(\Sel_\bullet(K_\infty,A)) = 1.$$
	\end{cor}
	\begin{proof}
		Under the assumption \eqref{hyp:nontors}, for all but finitely many characters $\alpha$, we have $\bka^\bullet_1(\alpha)\ne 0$. The system $\{\bka^\bullet_n(\alpha)\}_{n\in \cN}$ therefore satisfies the hypotheses of Theorem \ref{thm:howard} by Proposition \ref{prop:kappa-kolyvagin}, which implies
		\[
			\rank_\cO \left(\Sel_\bullet(K,T_\alpha)\right) = 1 = \corank_\cO\left(\Sel_\bullet(K,A_\alpha)\right)
		\]
		for all but finitely many $\alpha$. The corollary then follows from the fact that, if $M$ is a finitely generated $\Lambda$-module, then
		\[
			\rank_{\Lambda}M = \rank_\cO M(\alpha)
		\]
		for almost all $\alpha$.
	\end{proof}

\section{Kolyvagin's conjecture}This section is devoted to the proof of Theorem~\ref{thmA}.
Let $\bullet\in \{\sharp,\flat\}$. We begin with some preliminaries on the compact Selmer groups.

    \subsection{Facts about $\Sel_\bullet(K,\TT)$}
    \begin{lem}\label{lem:pseudo-null}
        Assume that \eqref{hyp:nontors} and \eqref{hyp:irr} hold. For all but finitely many $\alpha$, the $\fP_\alpha$-torsion of $\Sel_\bullet(K,\TT)/\Lambda\cdot\bka_1^\bullet$ is finite and has cardinality bounded indepedent of $\alpha$.
    \end{lem}
    \begin{proof}
         Corollary \ref{cor:rank-sel} tells us that $\rank_\Lambda \Sel_\bullet(K,\TT) = 1$. Hence there exists a pseudo-null submodule $Q'$ of $\Sel_\bullet(K,\TT)/\Lambda\bka_1^\bullet$ such that there exists a finite collection of polynomials $f_i\in \Lambda$ and an exact sequence
         \[
            0\to Q' \to \Sel_\bullet(K,\TT)/\Lambda\bka_1^\bullet \to \bigoplus_i \Lambda/(f_i).
         \]
         For all but finitely many $\alpha$, the last term have no $\fP_\alpha$-torsion, so the $\fP_\alpha$-torsion of the middle is isomorphic to $Q'[\fP_\alpha]$, which is the fixed finite group $(Q')^{\gamma = 1}$ when excluding finitely many $\alpha$'s.
    \end{proof}

    Next we will need a control theorem for compact Selmer groups:
    \begin{prop}\label{prop:control-compact}
        For all but finitely many $\alpha$, the specialization map $\Sel_\bullet(K,\TT)\to \Sel_\bullet(K,T_{\alpha})$ has finite cokernel whose cardinality is bounded independent of $\alpha$.
    \end{prop}
    To facilitate the proof we introduce the following group:
    \[
        S_\bullet(K,\TT) = \ker\left(
		H^1(K_\Sigma/K,\TT)\to \prod_{w\mid Np} 
		\frac{H^1(K_w,\TT)}{H^1_\bullet(K_w,\TT)+ \fP_{\alpha^{-1}}H^1(K_w,\TT)}
		\right).
    \]
    Note that $S_\bullet(K,\TT)$ contains $\fP_{\alpha^{-1}} H^1(K_\Sigma/K,\TT)$. 
    \begin{lem}\label{lem:pre-control-compact}
        For all but finitely many $\alpha$, the natural map $S_\bullet(K,\TT)\to \Sel_\bullet(K,T_\alpha)$ has finite cokernel whose cardinality is bounded independent of $\alpha$.
    \end{lem}
    \begin{proof}
        	Denote by $\phi$ the specialization map $S_\bullet(K,\TT)\to \Sel_\bullet(K,T_\alpha)$. For $W=T_\alpha$ or $\TT$, put
		\[
            \cP^\Sigma_\bullet(K,W) = \prod_{w\in \Sigma} H^1_{/\bullet}(K_w,W).
		\]
        We consider a diagram with exact rows (the exactness of the first row is by the construction of $S_\bullet(K,\TT)$):
		\begin{align}\label{diag:Sel-T}
			\begin{tikzcd}[ampersand replacement= \&]
				\&S_\bullet(K,\TT) \ar[r] \ar[d,"\phi"]
				\& H^1(K_\Sigma/K,\TT)/\fP_{\alpha^{-1}} \ar[r] \ar[d,"h'"]
				\& \cP^\Sigma_\bullet(K,\TT)/\fP_{{\alpha^{-1}}} \ar[d,"s"]\\
				0 \ar[r] \&\Sel_\bullet(K,T_\alpha) \ar[r] 
				\& H^1(K_\Sigma/K,T_\alpha) \ar[r] 
				\& \cP^\Sigma_\bullet(K,T_\alpha);
			\end{tikzcd}
		\end{align}
		By the snake lemma, the desired properties of $\coker(\phi)$ would follow from the combination of the following assertions:
		\begin{enumerate}
			\item[(a)] $s$ is injective;
			\item[(b)] $|\coker(h')|$ is bounded independent of $\alpha$ for all but finitely many $\alpha$.
		\end{enumerate}
		We first establish (a): We can break down $s$ into local maps
		\[
		s_{w}: H^1_{/\bullet}(K_w,\TT)/\fP_{\alpha^{-1}} \to H^1_{/\bullet}(K_w,T_\alpha)
		\]
		for each place $w$ of $K$, and we will show $\ker(s_w) = 0$ individually. Suppose first $w\mid p$, then the corestriction
		\[
		H^1(K_w,\TT)/\fP_{\alpha^{-1}} \xrightarrow{\sim} H^1(K_w,T_\alpha)
		\]
		is an isomorphism by taking the Pontryagin duals. Hence,
		\[
		H^1_{/\bullet}(K_w,\TT)/\fP_{\alpha^{-1}} = \frac{H^1(K_w,\TT)/\fP_{\alpha^{-1}}}{\im(H^1_\bullet(K_w,\TT)/\fP_{\alpha^{-1}})} \simeq H^1_{/\bullet}(K_w,T_\alpha).
		\]
		Thus, $s_w$ is an isomorphisms if $w\mid p$. Suppose $w\mid N$, then we have $H^1_{/\f}(K_w,\TT) = 0$ by \cite[\S2.2.4]{perrin-riou:theorie-iwasawa-hauteurs}. Hence $\ker(s_w) = 0$ for such $w$ also.
		
		We now turn our attention to (b). Take the $K_\Sigma/K$ cohomology of the short exact sequence
		\[
		0\to \TT\xrightarrow{\times(\gamma - \alpha(\gamma))}
		\TT \to T_\alpha\to 0,	
		\]
		we find $\coker(h') \simeq H^2(K_\Sigma/K,\TT)[\fP_{\alpha^{-1}}]$. Since $H^2(K_\Sigma/K,\TT)$ is a finitely generated $\Lambda$-module, we see that if $P$ is the maximal pseudo-null submodule of $H^2(K_\Sigma/K,\TT)$, $H^2(K_\Sigma/K,\TT)[\fP_{\alpha^{-1}}] = P^{\Gamma}$  for all but finitely many $\alpha$'s.
    \end{proof}

    \begin{proof}[Proof of Proposition \ref{prop:control-compact}]
        The game here is to show that $\im(\Sel_\bullet(K,\TT))\subset H^1(K,\TT)/\fP_{\alpha^{-1}}$ is a submodule of $S_\bullet(K,\TT)/\fP_{\alpha^{-1}}H^1(K,\TT)$ with index bounded independent of $\alpha$. Granting this, since the map $S_\bullet(K,\TT) \to \Sel_\bullet(K,T_\alpha)$ vanishes on $\fP_{\alpha^{-1}}H^1(K,\TT)$, the result then follows from Lemma \ref{lem:pre-control-compact}.

        We now turn to the claim about the index. Recall that the subgroup $S_\bullet(K,\TT)\subset H^1(K_\Sigma/K,\TT)$ enjoys the property
		\begin{align*}
			S_\bullet(K,\TT)/\fP_{\alpha^{-1}}H^1(K_\Sigma/K,\TT) = \ker\left(H^1(K_\Sigma/K,\TT)/\fP_{\alpha^{-1}} \to 
			\cP^\Sigma_\bullet(K,\TT)/\fP_{\alpha^{-1}} \right).
		\end{align*}
		On the other hand, applying the $\fP_{\alpha^{-1}}$-torsion functor on the exact sequence
		\begin{align*}
			0\to \Sel_\bullet(K,\TT)\to H^1(K_\Sigma/K,\TT)\to \cG^\Sigma_\bullet(K,\TT)\to 0,
		\end{align*}
		we find that
		\begin{align}\label{eq:Sel-Palpha}
			\cG^\Sigma_\bullet(K,\TT)[\fP_{\alpha^{-1}}]\to \Sel_\bullet(K,\TT)/\fP_{\alpha^{-1}}\to H^1(K_\Sigma/K,\TT)/\fP_{\alpha^{-1}} \to \cG^\Sigma_\bullet(K,\TT)/\fP_{\alpha^{-1}}
		\end{align}
		is exact. Note that
		\begin{align*}
			\cG^\Sigma_\bullet(K,\TT)[\fP_{\alpha^{-1}}] \subseteq \cP^\Sigma_\bullet(K,\TT)[\fP_{\alpha^{-1}}] = \prod_{w\mid p} H^1_{/\bullet}(K_w,\TT)[\fP_{\alpha^{-1}}],
		\end{align*}
		where the last equality is by the coincidence $ H^1(K_w,\TT) = H^1_\f(K_w,\TT)$ when $w\mid N$ by \cite[\S2.2.4]{perrin-riou:theorie-iwasawa-hauteurs}. Since the Coleman map has image in $\Lambda\otimes_{\Z_p} K_w$, we see that $H^1_{/\bullet}(K_w,\TT)[\fP_{\alpha^{-1}}] = 0$, and so the exactness of \eqref{eq:Sel-Palpha} gives
		\[
		\Sel_\bullet(K,\TT)/\fP_{\alpha^{-1}} = 
		\ker\left(H^1(K_\Sigma/K,\TT)/\fP_{\alpha^{-1}} \to 
			\cG^\Sigma_\bullet(K,\TT)/\fP_{\alpha^{-1}} \right).
		\]
		Thus, for our purpose, it suffices to prove that
        \[
            \ker(\cG^\Sigma_\bullet(K,\TT)/\fP_{\alpha^{-1}}\to \cP^\Sigma_\bullet(K,\TT)/\fP_{\alpha^{-1}})
        \]
        is finite and bounded independent of $\alpha$'s.
        
        Put $H = \cP^\Sigma_\bullet(K,\TT)/\cG^\Sigma_\bullet(K,\TT)$. By the Poitou--Tate duality, $H$ embeds into $\Sel_\bullet(K_\infty,A)^\vee$. As such, by the structure theorem of finitely-generated $\Lambda$-modules, $H[\fP_{\alpha^{-1}}]$ is finite for all but finitely many $\alpha$, and has cardinality bounded independent of $\alpha$. Since
		\begin{align*}
			H[\fP_{\alpha^{-1}}] \to \cG^\Sigma_\bullet(K,\TT)/\fP_{\alpha^{-1}} \to \cP^\Sigma_\bullet(K,\TT)/\fP_{\alpha^{-1}}
		\end{align*}
		is exact, this gives the desired finiteness.
    \end{proof}
	
    \subsection{Kolyvagin's conjecture via the divisibility argument}
    
    Denote by $\ind(\bka_1^\bullet(\alpha))$ the length of the $\cO$-module $\Sel_\bullet(K,T_\alpha)/\cO\cdot \bka_1^\bullet(\alpha)$. We introduce the following reverse inclusion of Conjecture \ref{conj:main}:
    \begin{align}\tag{incl.}\label{hyp:weak-IMC}
        \text{There exists $\mu\in \Z$ such that } \varpi^\mu\Char(\Sel_\bullet(K_\infty,A)^\vee_\tors) \subseteq \Char(\Sel_\bullet(K,\TT)/\Lambda\cdot\bka_1^\bullet)^2.
    \end{align}

\begin{prop}\label{prop:second-ineq}
    Assume \eqref{hyp:nontors}, \eqref{hyp:irr}, and \eqref{hyp:weak-IMC}. There exists a constant $C$ such that for all but finitely many $\alpha$, we have
    \[
        |\cO/\varpi|^{2\ind(\bka_1^\bullet(\alpha))} \le  C |\cO/\varpi|^\mu |\Sha_\bullet(K,A_\alpha)| \prod_{w\mid N^+} c_w^{(p)}.
    \]
    Furthermore, the equality holds when the inclusion in \eqref{hyp:weak-IMC} is an equality.
\end{prop}
\begin{proof}
    Let $\cH_\bullet$ be the characteristic polynomial of $\Sel_\bullet(K,\TT)/\Lambda\bka^\bullet_1$. By our assumption, $(\cH_\bullet)^2$ divides $\varpi^\mu\cF_\bullet$ (recall $\cF_\bullet$ is the characteristic polynomial of $\Sel_\bullet(K,A_\infty)^\vee$). As in the proof of Lemma \ref{lem:euler-char}, we find
    \begin{align}
        |\cO/\cH_\bullet(\alpha(\gamma)-1)| &= \frac{\left|\Sel_\bullet(K,\TT)/(\fP_{\alpha^{-1}}\Sel_\bullet(K,\TT) + \Lambda\bka_1^\bullet)\right|}{\left|(\Sel_\bullet(K,\TT)/\Lambda\bka_1^\bullet)[\fP_{\alpha^{-1}}]\right|}\\
        &\ge \frac{1}{C_1 C_2}\left|\frac{\Sel_\bullet(K,T_\alpha)}{\cO\cdot\bka_1^\bullet(\alpha)}\right| \\
        &=\frac{1}{C_1C_2}|\cO/\varpi|^{\ind(\bka_1^\bullet(\alpha))}.
    \end{align}
    Here, the constant $C_1$ is from Lemma \ref{lem:pseudo-null} and $C_2$ is from Proposition \ref{prop:control-compact}, neither of which depends on $\alpha$. The proof is then concluded by invoking Theorem \ref{thm:euler-char} and putting $C = (C_1C_2)^2$.
\end{proof}

Following \cite{BLV}, we refer to the situation
\begin{align}\tag{exc.}\label{hyp:exc}
    p\text{ is inert in $K$ and }\bullet = \sharp
\end{align}
the \textit{exceptional} case. Otherwise, we shall say that we are in the \textit{non-exceptional} case:
\begin{align}\tag{nonexc.}\label{hyp:nonexc}
    \text{Either $p$ is split in $K$ or }\bullet = \flat
\end{align}

\begin{thm}\label{thm:kolyvagin}
    Let $\bullet \in \{\sharp,\flat\}$, and assume that \eqref{hyp:nontors} and \eqref{hyp:irr} hold. In the non-exceptional case \eqref{hyp:nonexc}, suppose \eqref{hyp:weak-IMC} is valid. Then, for all $e\in \Z_{>0}$, there exists $n\in \cN^{(e)} = \cN(\cL^{(e)})$ such that $\kappa_n^\Heeg \not\in H^1(K,T/I_n)$. In the remaining exceptional situation \eqref{hyp:exc}, the inclusion in \eqref{hyp:weak-IMC} fails. 
\end{thm}
The key construction relies on the following divisibility lemma.
\begin{lem}\label{lem:index}
    Let $\bullet\in \{\sharp,\flat\}$ and $n\in \cN$. Suppose $\bka^\bullet_n(\1) = 0$ in $H^1(K, T/I_n)$. Then for all $\alpha\equiv 1\bmod \varpi^m$ for some $m\in \Z_{>0}$, we have $\bka_n^\bullet(\alpha) \in \varpi^m H^1(K,T/I_n)$.
\end{lem}
\begin{proof}
    
	Consider $\tilde\bka^\bullet_n \in H^1(K[n],\TT)$ from \S\ref{subsec:signed-heegner}. Note the following chain
	\[
		\begin{tikzcd}
			H^1(K[n], \TT) \ar[r,"a"] & H^1(K[n],\TT/I_n) \ar[r,"b"] & H^1(K[n],T/I_n)\\
			\tilde\bka_n^\bullet \ar[r,mapsto] & \tilde\bka_n^\bullet \bmod I_n \ar[r,mapsto] & \tilde\bka_n^\bullet(\1) \bmod I_n
		\end{tikzcd}
	\]
	Note that by construction, $\tilde\bka_n^\bullet \bmod I_n$ is invariant under $\cG(n)$-action. Our assumption then says that $\tilde\bka_n^\bullet \bmod I_n$ is in $\ker(b)$, which is $(\gamma -1)H^1(K[n],\TT/I_n)$. We contend that
	\[
		(\gamma -1)H^1(K[n],\TT/I_n)\cap H^1(K[n],\TT/I_n)^{\cG(n)} = (\gamma-1) H^1(K[n],\TT/I_n)^{\cG(n)}.
	\]
	Indeed, suppose $x = (\gamma -1)y$ is in the left hand side for some $y\in H^1(K[n],\TT/I_n)$. Then for all $\sigma \in \cG(n)$, as $\sigma(\gamma-1)y = (\gamma-1)(\sigma y)$, we have $y^\sigma - y$ is killed by $\gamma -1$. On the other hand, we know that
	\[
		0\to \TT/I_n \xrightarrow{\times (\gamma-1)} \TT/I_n \to T/I_n \to 0
	\]
	is exact, since as an $\Lambda$-module, the first can be non-canonically identified with $(\Lambda^\iota/I_n)^{\rank_\cO T}$, which has no $(\gamma -1)$-torsion. We have thus a surjection derived from taking the group cohomology:
	\[
		(T/I_n)^{G_{K[n]}}\to H^1(K[n],\TT/I_n)[\gamma -1]\to 0.
	\]
	The first item is zero by Lemma \ref{lem:T-no-invariant}, so $H^1(K[n],\TT/I_n)$ has no $(\gamma-1)$-torsion, whereby $y$ is invariant under $\cG(n)$. This proves the contention.
	
	Consider now the $\alpha$-specialization
	\[
	\begin{tikzcd}
		H^1(K[n],\TT) \ar[r,"a"] & H^1(K[n],\TT/I_n) \ar[r,"b_\alpha"] & H^1(K[n],T_\alpha/I_n)\\
		\tilde\bka_n^\bullet \ar[r,mapsto] & \tilde\bka_n^\bullet \bmod I_n \ar[r,mapsto] & \tilde\bka_n^\bullet(\alpha) \bmod I_n
	\end{tikzcd}
	\]
	As $\tilde\bka_n^\bullet\bmod I_n \in (\gamma-1)\cdot H^1(K[n],\TT/I_n)^{\cG(n)}$, we see that
    \[
    \tilde\bka_n^\bullet(\alpha) \bmod I_n \in (\alpha(\gamma)-1)H^1(K[n],T_\alpha/I_n)^{\cG(n)}.
    \]
    An argument similar to Corollary \ref{cor:iso-rest} then shows $\bka_n^\bullet(\alpha) \in \varpi^m H^1(K,T_\alpha/I_n)$.
\end{proof}
\begin{proof}[Proof of Theorem \ref{thm:kolyvagin}]
    We first deal with the non-exceptional case, and suppose, for the sake of contradiction, that $\kappa_n^\Heeg=0$ for all $n\in \cN^{(e)}$. 

    For all $n\in \cN^{(e)}$, applying Lemma \ref{lem:cong-class-alpha} with the trivial character $\1$ and $\varpi^m\cO = I_n\cO$, we find $\bka_n^\bullet(\1) = 0$. Fix $m\gg 0$ and $\alpha\equiv 1\bmod \varpi^m$. For all $n\in \cN^{(e)}$, by Lemma \ref{lem:index}, we conclude that $\bka_n^\bullet(\alpha)\in \varpi^m H^1(K,T_\alpha/I_n)$. As we have assumed that $\bka_1^\bullet$ is not $\Lambda$-torsion, there are only finitely many characters $\alpha$ for which $\bka_1^\bullet(\alpha) = 0$. Avoiding these characters, we can choose $\breve \kappa_n^\bullet \in H^1(K,T_\alpha/I_n)$ such that $\varpi^m \breve \kappa_n^\bullet = \bka_n^\bullet(\alpha)$ for all $n\in \cN^{(e)}$, and $\breve \kappa_1^\bullet \ne 0$. So, by Proposition \ref{prop:kappa-kolyvagin}, we obtain a system $\{\breve \kappa_n^\bullet\}_{n\in \cN^{(e)}}$ for which Theorem \ref{thm:howard} applies. This gives 
		\[
		|\Sha_\bullet(K,A_\alpha)| \le |\cO/\varpi|^{2\ind(\breve\kappa_1^\bullet)} = |\cO/\varpi|^{2\ind(\bka_1^\bullet(\alpha))-2m}.
		\]
		By Proposition \ref{prop:second-ineq}, we find
		\[
		|\Sha_\bullet(K,A_\alpha)| \le C|\cO/\varpi|^{-2m} |\cO/\varpi|^\mu |\Sha_\bullet(K,A_\alpha)| \prod_{w\mid N^+}c_w^{(p)},
		\]
		so
		\[
		2m \le C' + \mu + f(F/\Q_p)^{-1}\sum_{p\mid N^+}\ord_p(c_w^{(p)})
		\]
		for some constant $C'$ independent of $m$ and $f(F/\Q_p)$ the inertia degree of $F/\Q_p$. Since this holds for arbitrarily large $m$, we derive a contradiction.
		
		Concerning the exceptional scenario, note that by Remark \ref{rem:kappa-sharp}, we have $\bka_n^\sharp(\1)= 0$ for all $n\in \cN$. We can thus construct a system $\breve\kappa_n^{\sharp}$ in the above manner. The inclusion \eqref{hyp:weak-IMC} would give us the same contradiction as above, and hence it cannot be valid.
\end{proof}

\bibliographystyle{amsalpha}
\bibliography{references}
\end{document}